\documentclass[a4paper,12pt,reqno]{amsart}
\usepackage{amssymb}
\usepackage{amsmath}
\usepackage{graphics}
\usepackage{color}
\usepackage{multirow}
\usepackage{pstricks}
\usepackage{graphicx,verbatim,psfrag,mathdots}
\usepackage[colorlinks=true]{hyperref}
\usepackage{ytableau}
\usepackage{mathabx}
\usepackage{rotating}
\usepackage{musicography}

\graphicspath{{figs/}}

\usepackage{tikz}
\usetikzlibrary{arrows}

\newcommand{\swne}{`` \tikz \draw (0,0) -- (.6,.3); ''}
\newcommand{\senw}{`` \tikz \draw (.6,0) -- (0,.3); ''}
\newcommand\mystrut{\rule{0pt}{8pt}}

\newcommand{\floor}[1]{\lfloor #1 \rfloor}

\newcommand{\x}{\bar{x}}
\newcommand{\ws}[1]{\operatorname{w}_{\operatorname{sp}}(#1)}
\newcommand{\oo}{so^{\mathrm{odd}}}
\renewcommand{\oe}{o^{\mathrm{even}}}
\renewcommand{\sp}{sp}
\newcommand{\woo}[1]{\operatorname{w}_{\operatorname{so}}(#1)}
\newcommand{\woe}[1]{\operatorname{w}_{\operatorname{o}}^{\operatorname{even}}(#1)}
\newcommand{\T}{\operatorname{T}}
\newcommand{\ST}{\operatorname{ST}}
\newcommand{\DT}{\operatorname{DT}}
\newcommand{\SDT}{\operatorname{SDT}}
\newcommand{\htm}[2]{\operatorname{HT\text{--}}^{#1}_{#2}}
\newcommand{\hhtm}[2]{\widehat{\operatorname{HT\text{--}}}^{#1}_{#2}}
\newcommand{\htp}[2]{\operatorname{HT+}^{#1}_{#2}}
\newcommand{\hhtp}[2]{\widehat{\operatorname{HT+}}^{#1}_{#2}}

\newcommand{\n}{\operatorname{n}}
\newcommand{\m}{\operatorname{M}}

\newcommand{\w}{\operatorname{w}}

\newcommand{\sgn}{\operatorname{sgn}}

\newcommand{\id}{\operatorname{Id}}

\renewcommand{\nu}{\operatorname{nu}}

\newtheorem{thm}{Theorem}
\newtheorem{lem}[thm]{Lemma}

\newtheorem{prop}[thm]{Proposition}
\newtheorem{rem}[thm]{Remark}
\newtheorem{defi}[thm]{Definition}

\newtheorem{example}[thm]{Example}

\numberwithin{equation}{section}
\numberwithin{thm}{section}

\DeclareMathOperator{\OP}{\mathcal{OP}}
\DeclareMathOperator{\SP}{\mathcal{SP}}
\DeclareMathOperator{\SOP}{\mathcal{SOP}}
\DeclareMathOperator{\SPT}{\mathcal{SPT}}
\DeclareMathOperator{\EOT}{\mathcal{EOT}}
\DeclareMathOperator{\OOT}{\mathcal{OOT}}

\begin{document}

\title{Bijective proofs of skew Schur polynomial factorizations}

\author[Arvind Ayyer]{Arvind Ayyer}
\address{Arvind Ayyer, Department of Mathematics, Indian Institute of Science, Bangalore - 560012, India}
\email{arvind@iisc.ac.in}
\author[Ilse Fischer]{Ilse Fischer}
\address{Ilse Fischer, Fakult\"{a}t f\"{u}r Mathematik, Universit\"{a}t Wien, Oskar-Morgenstern-Platz 1, 1090 Wien, Austria}
\email{ilse.fischer@univie.ac.at}

\begin{abstract} 
In a recent paper, Ayyer and Behrend present for a wide class of partitions factorizations of Schur polynomials with an even number of variables where half of the variables are the reciprocals of the others into symplectic and/or orthogonal group characters, thereby generalizing results of Ciucu and Krattenthaler for rectangular shapes. Their proofs proceed by manipulations of determinants underlying the characters. The purpose of the current paper is to provide bijective proofs of such factorizations. The quantities involved have known combinatorial interpretations in terms of Gelfand-Tsetlin patterns of various types or half Gelfand-Tsetlin patterns, which can in turn be transformed into perfect matchings of weighted trapezoidal honeycomb graphs. An important ingredient is then Ciucu's theorem for graphs with reflective symmetry. However, before being able to apply it, we need to employ a certain averaging procedure in order to achieve symmetric edge weights. This procedure is based on a ``randomized'' bijection, which can however also be turned into a  classical bijection. For one type of Schur polynomial factorization, we also need an additional graph operation that almost doubles the underlying graph. Finally, our combinatorial proofs reveal that the factorizations under consideration can in fact also be generalized to skew shapes as discussed at the end of the paper.
\end{abstract}

\maketitle

\section{Introduction}

\emph{Schur polynomials} $s_\lambda(x_1,\ldots,x_n)$ are central objects in algebraic combinatorics with various beautiful properties and numerous applications. In representation theory, they are the irreducible characters of polynomial representations of the \emph{general linear group} $GL_n(\mathbb{C})$. Ayyer and Behrend \cite[Theorem~1]{AyyBeh18} showed that for two families of partitions, Schur polynomials with $2n$ variables factorize into characters of other classical groups when specializing such that half of the variables are the reciprocals of the others. The two families of partitions are 
\begin{equation}(\lambda_1+1,\lambda_2+1, \dots, \lambda_n + 1, 
 - \lambda_n,-\lambda_{n-1},\dots, - \lambda_1)+ \lambda_1 \quad 
\end{equation} 
and 
\begin{equation}
(\lambda_1, \dots, \lambda_n, 
- \lambda_n, \dots, - \lambda_1)+\lambda_1,
\end{equation}
where $\lambda=(\lambda_1,\ldots,\lambda_n)$ is a partition, allowing here and throughout the whole paper zero parts in a partition and ``$+\lambda_1$'' means that we add $\lambda_1$ to each part.
An illustration of the families is provided in Figure~\ref{part}.
\begin{figure}
\scalebox{0.4}{
\psfrag{n}{\Large$n$}
\psfrag{L}{\Huge$\lambda$}
\psfrag{L1}{\Large$\lambda_1$}
\psfrag{L2}{\Huge\begin{turn}{180}$\lambda$ \end{turn}}
\psfrag{1}{\Large$1$}
\includegraphics{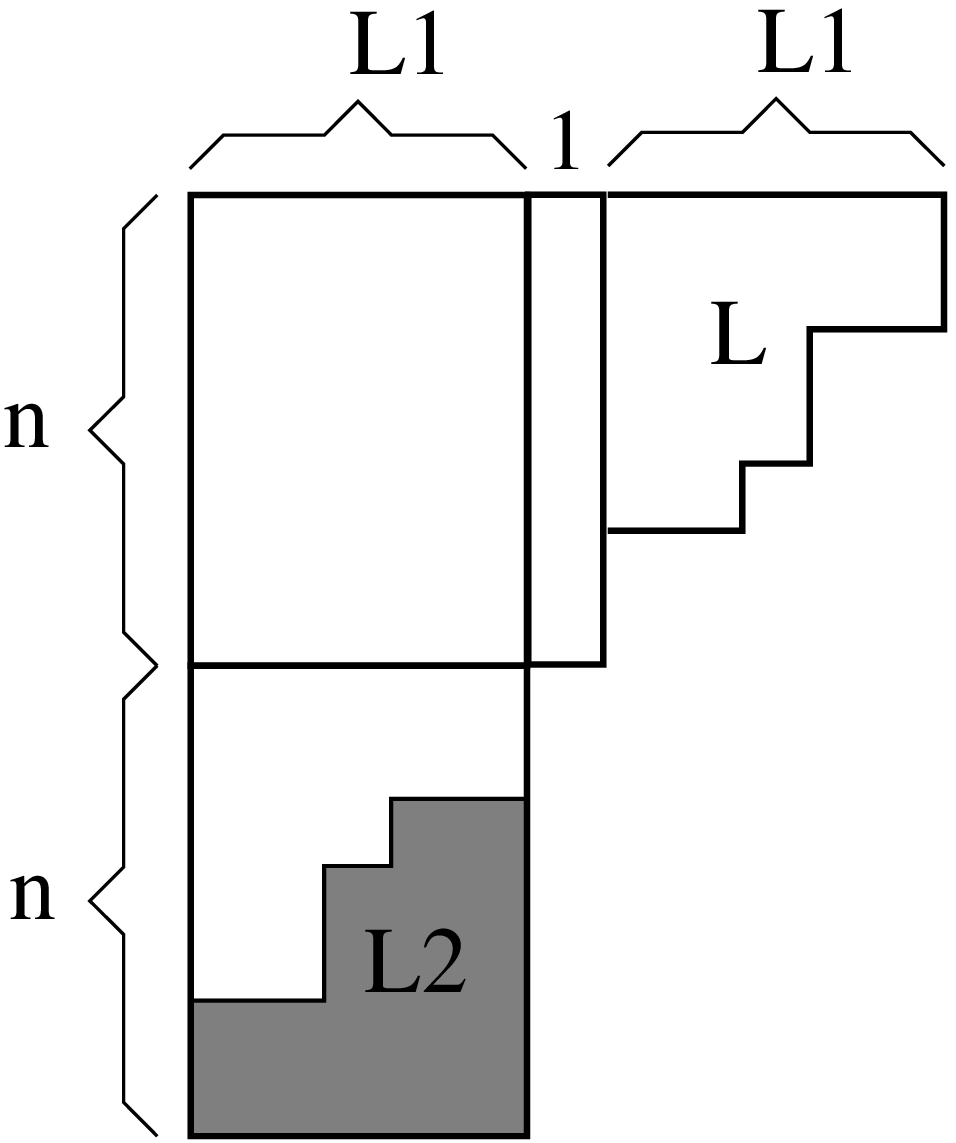}}
\hspace{2cm}
\scalebox{0.4}{
\psfrag{n}{\Large$n$}
\psfrag{L}{\Huge$\lambda$}
\psfrag{L1}{\Large$\lambda_1$}
\psfrag{L2}{\Huge\begin{turn}{180}$\lambda$ \end{turn}}
\psfrag{1}{\Large$1$}
\includegraphics{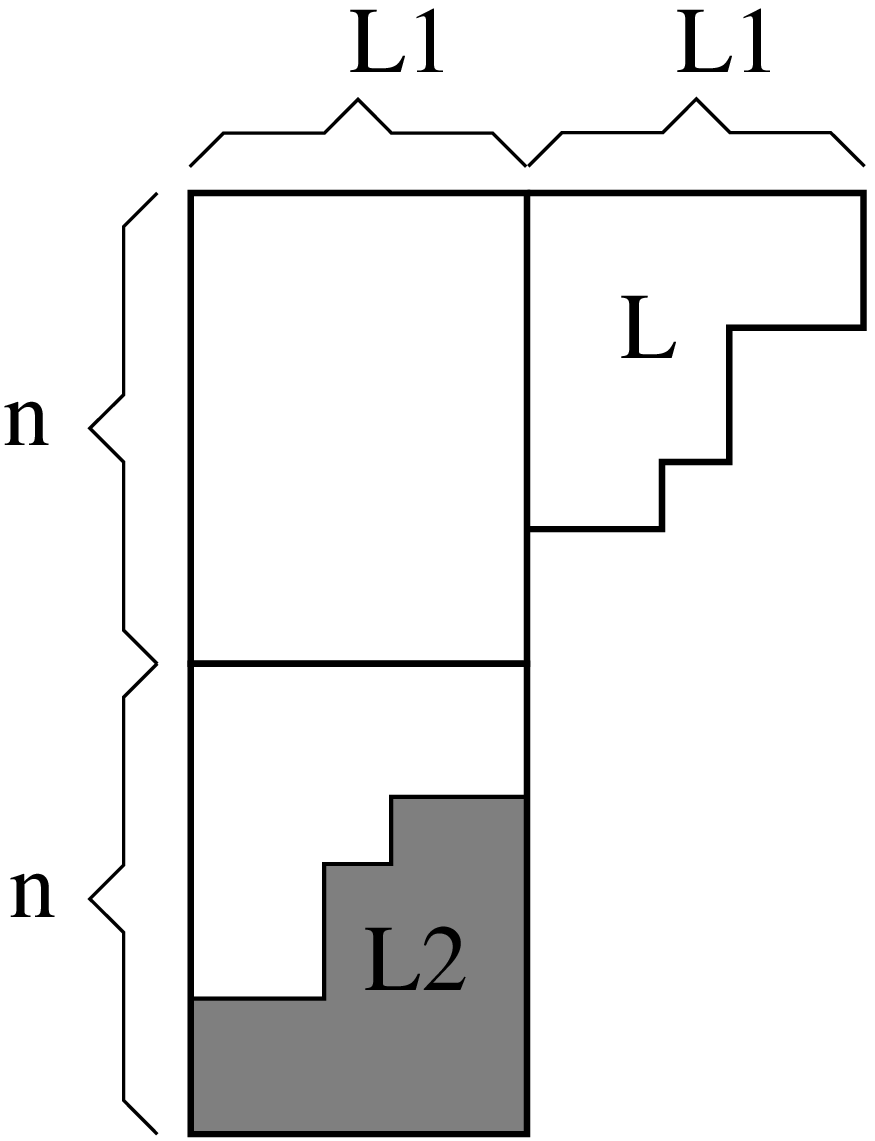}}
\caption{\label{part} The two families of partitions}
\end{figure}

In order to state the precise result, we use the following standard notation for characters of classical groups. Let $n$ be a positive integer and $\lambda=(\lambda_1,\ldots,\lambda_n)$ be a partition with at most $n$ non-zero parts, adding trailing zeros if necessary. Recall the following well-known formula for Schur polynomials.
\begin{equation}
s_{\lambda}(x_1,\ldots,x_n) = \frac{\det_{1 \le i,j \le n} \left( x_i^{\lambda_j+n-j} \right)}{\det_{1 \le i,j \le n} \left( x_i^{n-j} \right)}
\end{equation} 
The other characters appearing in this paper are the following. Throughout the article, we set $\bar{x} = x^{-1}$.
\begin{itemize} 
\item  {\em Symplectic characters} are the irreducible characters of the \emph{symplectic group} $Sp_{2n}(\mathbb{C})$, and they are given by 
\begin{equation}
sp_{\lambda}(x_1,\ldots,x_n)  = \frac{\det_{1 \le i,j \le n} \left( x_i^{\lambda_j+n-j+1} - \bar x_i^{\lambda_j+n-j+1}   \right)}{\det_{1 \le i,j \le n} \left( x_i^{n-j+1} - \bar x_i^{n-j+1}  \right)}. \end{equation} 
\item {\em Even orthogonal characters}  are the irreducible characters of the \emph{even orthogonal group} $O_{2n}(\mathbb{C})$, and they are given by 
\begin{equation}
\oe_\lambda(x_1,\ldots,x_n)  = (1 + [\lambda_n \neq 0]) \frac{\det_{1 \le i,j \le n} \left( x_i^{\lambda_j+n-j} + \bar x_i^{\lambda_j+n-j}   \right)}{\det_{1 \le i,j \le n} \left( x_i^{n-j} + \bar x_i^{n-j}  \right)},
\end{equation}
where we use the Iverson bracket, i.e., $[ S ]$ is $1$ if $S$ is true and $0$ if $S$ is false.
 \item {\em Odd orthogonal characters} are the irreducible characters of the \emph{special odd orthogonal group} $SO_{2n+1}(\mathbb{C})$, and they are given by
 \begin{equation}
 \oo_\lambda(x_1,\ldots,x_n)  = \frac{\det_{1 \le i,j \le n} \left( x_i^{\lambda_j+n-j+1/2} - \bar x_i^{\lambda_j+n-j+1/2}   \right)}{\det_{1 \le i,j \le n} \left( x_i^{n-j+1/2} - \bar x_i^{n-j+1/2}  \right)}, 
 \end{equation}  
 where $\delta$ is the Kronecker delta.
\end{itemize}
A \emph{half-integer} is an odd integer divided by $2$. A \emph{half-integer partition} is a finite weakly decreasing sequence of positive half-integers. In the case of the even orthogonal group, the character formula has a representation theoretic meaning as characters of spin covering groups when $\lambda$ is a half-integer partition. For more information see, e.g., the book by Fulton and Harris  \cite[Chap. 24]{FulHar91}.

The starting point for the research presented in this paper was to provide a combinatorial proof of the following theorem, which appeared in \cite[Theorem~1]{AyyBeh18} in a slightly different but equivalent form as also explained there: The case of \cite[Theorem~1]{AyyBeh18} in which all parts of $\lambda$ are half-integers corresponds to the first part of the theorem below  (see also \cite[Equation (19)]{AyyBeh18}), while the case in which all parts of $\lambda$ are integers corresponds to the second part
 (see also \cite[Equation (18)]{AyyBeh18}).

For a partition $\lambda$ and an integer or half-integer $\ell$, we denote by $\lambda + \ell$ the tuple obtained by adding $\ell$ to each part of $\lambda$, and we set
$\bar x = x^{-1}$.

\begin{thm}
\label{thm:fact1}
Let $n$ be a positive integer and $\lambda = (\lambda_1,\dots,\lambda_n)$ be a partition.
\begin{enumerate}
\item \label{item:part1}
For 
\begin{equation}
\widehat{\lambda} = (\lambda_1+1,\lambda_2+1, \dots, \lambda_n + 1, 
 - \lambda_n,-\lambda_{n-1},\dots, - \lambda_1)+ \lambda_1,
\end{equation}
we have 
\begin{equation}
\label{toprove}
s_{\widehat{\lambda}}(x_1,\x_1, \ldots,x_{n},\x_{n})=
\sp_{\lambda}(x_1,\ldots,x_{n}) \: \oe_{\lambda+1}(x_1,\ldots,x_{n}).
\end{equation}
\item  \label{item:part2}

For
\begin{equation}
\widehat{\lambda} = (\lambda_1, \dots, \lambda_n, 
- \lambda_n, \dots, - \lambda_1)+\lambda_1,
\end{equation}
we have 
\begin{equation}
\label{toprove2}
s_{\widehat{\lambda}}(x_1,\x_1,\ldots,x_{n},\x_{n})= \prod_{i=1}^{n} (x_i^{1/2}+\x_i^{1/2})^{-1} 
\oo_{\lambda}(x_1,\ldots,x_{n}) \: \oe_{\lambda+\frac{1}{2}}(x_1,\ldots,x_{n}).
\end{equation}
\end{enumerate}
\end{thm}

Special cases of this theorem were known earlier; those for rectangular shapes in \cite{CiuKra09} and those for double-staircase shapes were announced in \cite{AyyBehFis16,BehFisKon17}. The known proofs of Theorem ~\ref{thm:fact1} and its special cases all proceed by manipulations of determinants underlying the characters. In this paper, we interpret the characters as generating functions of Gelfand-Tsetlin patterns, or, equivalently, as matching generating functions of edge-weighted subgraphs of the hexagonal grid, which makes it possible to provide a combinatorial proof. The partitions 
$\widehat{\lambda}$ in Theorem~\ref{thm:fact1} are chosen in such a way that the graphs employ a vertical symmetry, which suggests the use of Ciucu's factorization theorem for graphs with reflective symmetry \cite{Ciu97}. However, since the edge weights of the graphs are not symmetric, we need to ``symmetrize'' the weights. This is accomplished by a certain averaging procedure. This procedure is most conveniently explained by what we call a ``randomized'' bijection (but it  can also be turned into a classical bijection). This procedure suffices to fully deal with the identity in Theorem~\ref{thm:fact1}(\ref{item:part1}). In the case of Theorem~\ref{thm:fact1}(\ref{item:part2}), this procedure results in graphs with symmetric edge weights except for edges incident with the symmetry axis. We resolve this problem by applying a certain graph operation that, in a sense, almost doubles the graph and then Ciucu's factorization theorem is applicable also in this case.

A merit of bijective proofs is often that they reveal more about the relation between two types of objects than ``just'' the fact that they are counted by the same numbers.
In our case, it actually reveals quite naturally the following generalization of Theorem~\ref{thm:fact1} to skew shapes. This was not obvious from the previous proofs as they are based on determinantal formulas for the group characters that do not generalize to skew shapes. For the definition of skew symplectic characters and of skew orthogonal characters as well as a discussion of their appearance in the literature as restrictions of straight characters to certain subgroups, we defer to Section~\ref{sec_skew}.

\begin{thm}
\label{skew_theo}
Let $m,n$ be non-negative integers with $m < n$, and $\mu=(\mu_1,\ldots,\mu_{m})$ and $\lambda=(\lambda_1,\ldots,\lambda_n)$ be partitions. 
\begin{enumerate} 
\item For 
\begin{align}
\widehat{\mu} &= (\mu_1+1,\mu_2+1, \dots, \mu_{m} + 1, 
 - \mu_{m},-\mu_{m-1},\dots, - \mu_1)+ \lambda_1, \\ 
\widehat{\lambda} & = (\lambda_1+1,\lambda_2+1, \dots, \lambda_n + 1, 
 - \lambda_n,-\lambda_{n-1},\dots, - \lambda_1)+ \lambda_1,
\end{align}
we have 
\begin{equation}
\label{toshow_skew}
s_{\widehat{\lambda}/\widehat{\mu}}(x_1,\x_1, \ldots,x_{n-m},\x_{n-m})=
\sp_{\lambda/\mu}(x_1,\ldots,x_{n-m}) \: \oe_{(\lambda+1)/(\mu+1)}(x_1,\ldots,x_{n-m}).
\end{equation}
\item For
\begin{align}
\widehat{\mu} &= (\mu_1,\mu_2, \dots, \mu_{m}, 
 - \mu_{m},-\mu_{m-1},\dots, - \mu_1)+ \lambda_1, \\ 
\widehat{\lambda} & = (\lambda_1,\lambda_2, \dots, \lambda_n, 
 - \lambda_n,-\lambda_{n-1},\dots, - \lambda_1)+ \lambda_1,
\end{align}
we have 
\begin{equation}
\label{toshow_skew2}
s_{\widehat{\lambda}/\widehat{\mu}}(x_1,\x_1,\ldots,x_{n-m},\x_{n-m})= \prod_{i=1}^{n-m} (x_i^{1/2}+\x_i^{1/2})^{-1} 
\oo_{\lambda/\mu}(x_1,\ldots,x_{n-m}) \: \oe_{\left(\lambda+\frac{1}{2}\right)/\left(\mu+\frac{1}{2}\right)}(x_1,\ldots,x_{n-m}).
\end{equation}
\end{enumerate}
In both statements, the skew Schur polynomial on the left-hand side has to be interpreted to be zero if the shape $\widehat{\mu}$ is not contained in $\widehat{\lambda}$, and the situation is similar for the characters appearing on the right-hand side. 
\end{thm}

It would be interesting to find representation-theoretic proofs of Theorems~\ref{thm:fact1} and \ref{skew_theo}.

\subsection*{Structure of the paper} The paper is organized as follows. In Section~\ref{patterns}, we establish graphical interpretations of the group characters appearing in Theorem~\ref{thm:fact1}. For Schur polynomials, this is fairly standard. For the other characters, we rely on the work of Proctor~\cite{Kin76,KinElS84,Pro94} where he provides combinatorial interpretations in terms of various types of half (Gelfand-Testlin) patterns. We then use the general idea from the Schur case to obtain equivalent graphical models in terms of honeycomb graphs. However, in these cases there are a few subtleties to take into account. In Section~\ref{sec_symmetrizing}, we then perform the above mentioned averaging procedure to achieve symmetric edge weights in the case of Theorem~\ref{thm:fact1}(\ref{item:part1}). The proof of this identity is then concluded in Section~\ref{sec_ciucu} using Ciucu's factorization result. As for the proof of 
Theorem~\ref{thm:fact1}(\ref{item:part2}), we introduce the above mentioned doubling operation in Section~\ref{sec_doubling} and then complete the proof again by using Ciucu's factorization result. In Section~\ref{sec_skew}, we deal with the case of skew Schur polynomials, where it will be seen that it is straightforward to generalize the proof for straight shapes. Building on previous work by Koike and Terada~\cite{KoiTer90} and others, we briefly discuss the previous appearance of the factors on the right-hand sides of \eqref{toshow_skew} and \eqref{toshow_skew2} in a representation theoretic setting. 

\subsection*{Conventions} Throughout the article, we set \begin{equation}\bar x = x^{-1}.\end{equation}
For a partition $\lambda=(\lambda_1,\ldots,\lambda_n)$, we always allow zero parts in a partition. 
Our graphs are edge-weighted in general and if the weight of an edge is not specified, it is 1.
Whenever we speak of a matching of a graph, we usually mean a perfect matching unless stated differently.

\section{Combinatorial interpretations of group characters as matching generating functions}
\label{patterns}

The purpose of this section is to provide combinatorial interpretations of the quantities in Theorem~\ref{thm:fact1} in terms of matching generating functions. 

\subsection{The general linear group}

For a positive integer $n$ and a partition $\lambda$, the associated \emph{Schur polynomial} $s_{\lambda}(x_1,\ldots,x_n)$ is known to be the generating function of \emph{semistandard tableaux of shape $\lambda$}\footnote{That is fillings of the Young diagram of shape $\lambda$ which are weakly increasing along rows and strictly increasing along columns.} with entries in $\{1,2,\ldots,n\}$ 
where the weight of a particular semistandard tableau $T$ is 
\begin{equation}
\label{tableau-weight}
x_1^\text{$\#$ of $1$'s in $T$} x_2^\text{$\#$ of $2$'s in $T$} \cdots 
x_n^\text{$\#$ of $n$'s in $T$}. 
\end{equation}
From this interpretation it is obvious that the Schur polynomial $s_{\lambda}(x_1,\ldots,x_n)$ vanishes if $\lambda$ has more than $n$ non-zero parts.
It is fundamental to our combinatorial proof to work with a different interpretation as a generating function which is in terms of the \emph{(perfect) matching generating function} of a certain subgraph of the \emph{hexagonal grid} of trapezoidal shape. The graphs relevant for our graphical model are the following.

\begin{defi}[The graph $\T_{n,k}$] Let $n,k$ be positive integers. The subgraph of the hexagonal grid that consists of $n-1$ centered rows of consecutive hexagons of lengths $k,k+1,\ldots,k+n-2$ with two edges added, one incident with the bottom vertex of the leftmost vertical edge and the other incident with the bottom vertex of the rightmost vertical edge, is denoted by $\T_{n,k}$. The degenerate case $\T_{1,k}$ consists of a zig-zag line with $2k+1$ vertices.
\end{defi}

See Figure~\ref{T68}(a) for a drawing of $\T_{6,8}$, where the two additional edges are marked in red. The graph $\T_{n,k}$ is a bipartite graph that has $n$ more vertices in one vertex class than in the other vertex class, thus it has no perfect matching. This can be changed by attaching $n$ vertical edges to a selection of $n$ of the $n+k$ bottommost vertices.  In addition, we also introduce edge weights.

\begin{figure}
\begin{tabular}{c c}
\scalebox{0.35}{\includegraphics{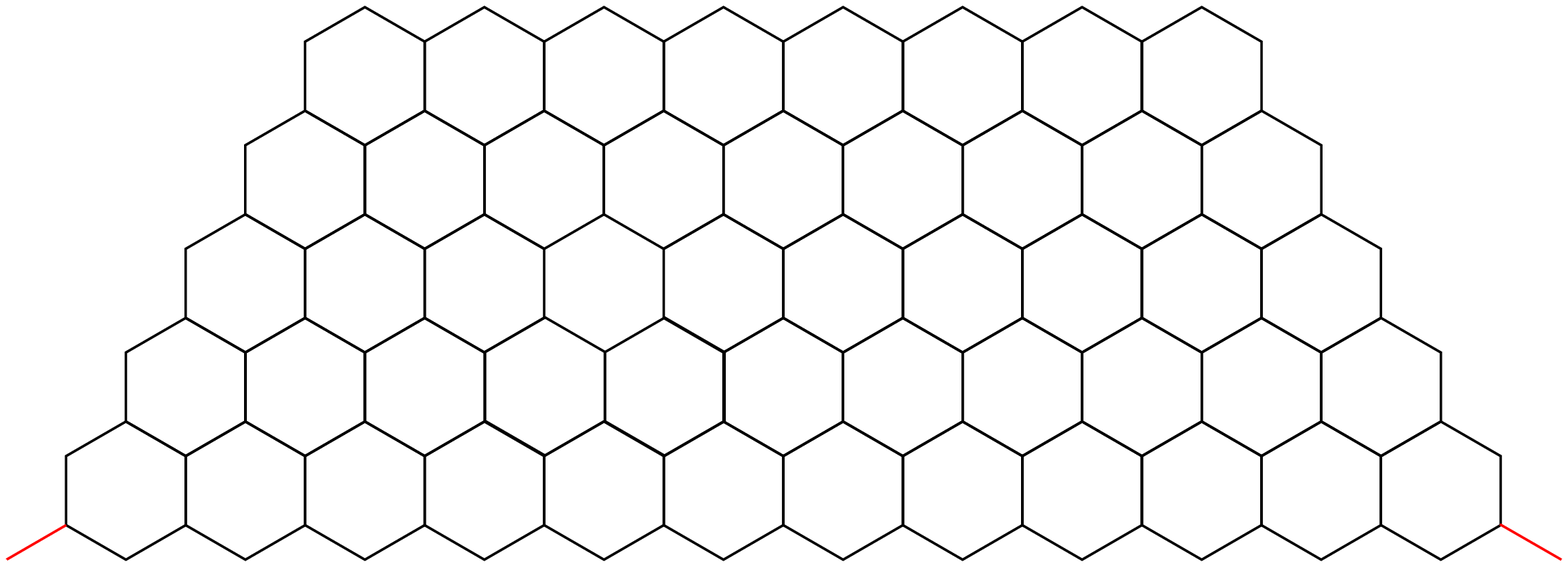}}
&
\scalebox{0.35}{
\psfrag{1}{\Large$1$}
\psfrag{2}{\Large$2$}
\psfrag{3}{\Large$3$}
\psfrag{4}{\Large$4$}
\psfrag{5}{\Large$5$}
\psfrag{6}{\Large$6$}
\includegraphics{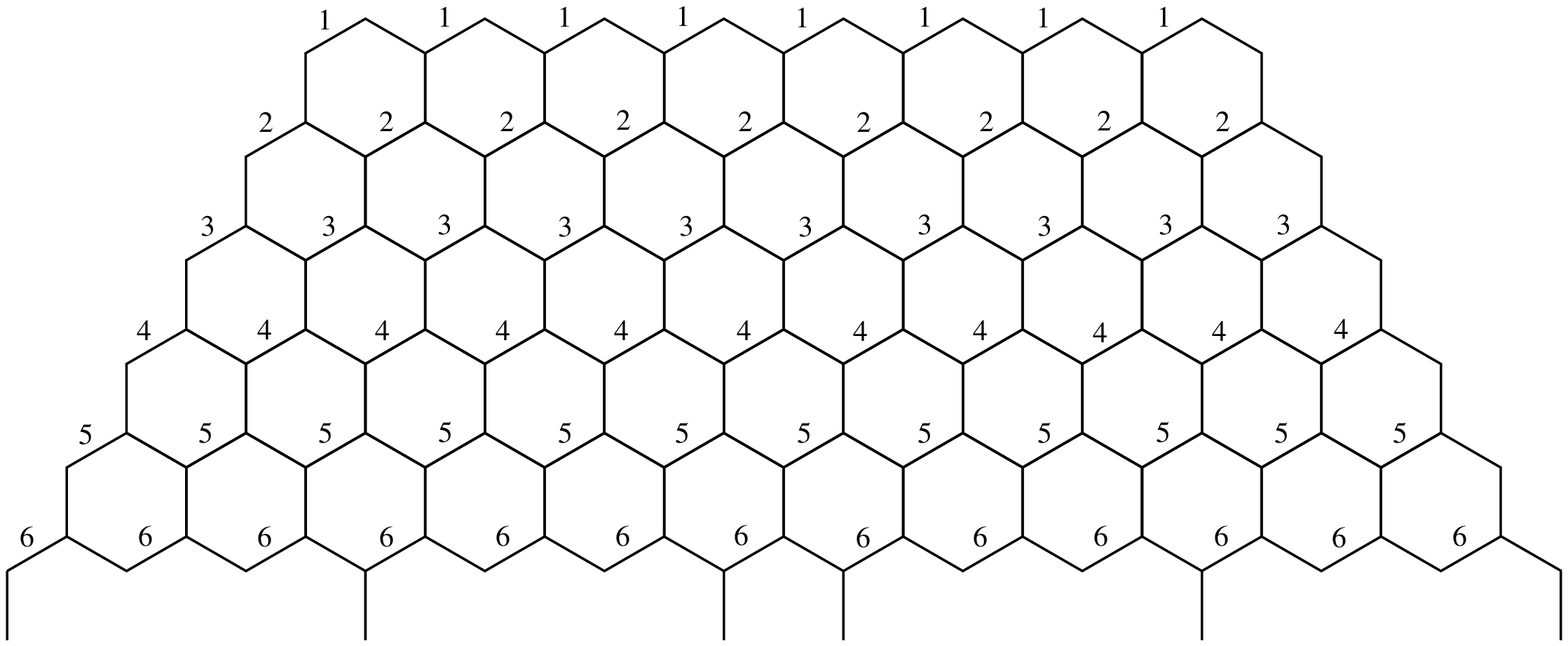}} \\
(a) & (b) 
\end{tabular}
\caption{\label{T68} The trapezoidal honeycomb graphs (a) $T_{6,8}$ and 
(b) $T_{6,8}^{1,4,7,8,11,14}(x_1,\ldots,x_6)$.}
\end{figure}

\begin{defi}[The graphs $\T_{n,k}^{\mathbf{p}}$ and $\T_{n,k}^{\mathbf{p}}(x_1,\ldots,x_n)$] Let $n,k$ be positive integers and $\mathbf{p}=(p_1,\ldots,p_n)$ be a sequence of integers with $1 \le p_1 < p_2 < \ldots < p_n \le n+k$.
\begin{enumerate} 
\item The graph $\T_{n,k}^{\mathbf{p}}$ is obtained from $\T_{n,k}$ by attaching vertical edges to the vertices in positions $p_1,p_2,\ldots,p_n$ at the bottom, where the bottommost vertices are numbered from left to right with $1,2,\ldots,n+k$.
\item The weighted graph $\T_{n,k}^{\mathbf{p}}(x_1,\ldots,x_n)$ is obtained from $\T_{n,k}^{\mathbf{p}}$ as follows: Each edge of type \swne (that is SW-NE edges) in row $i$ of zig-zag lines (counted from the top) carries the weight $x_i$, while all other edges have weight $1$. 
\end{enumerate}
\end{defi}

See Figure~\ref{T68}(b) for the graph $\T_{6,8}^{1,4,7,8,11,14}$.
The weights are also indicated in this figure, where $x_i$ is abbreviated as $i$, and, by our convention, edges have weight $1$ if no weight is indicated. As usual, the weight of a matching is the product of the weights of all edges that are contained in the matching, and the matching generating function is the sum of all matching weights. In general, the matching generating function is denoted by $\m(G)$, where $G$ is an edge-weighted graph, and $\mathcal{M}(G)$ denotes the set of all matchings of $G$. 

We are now in a position to state the different interpretation of $s_{\lambda}(x_1,\ldots,x_n)$ as a generating function. As noted above, the Schur polynomial is zero unless $\lambda$ has at most $n$ parts. For partitions with less than $n$ parts, it is convenient to fill up $\lambda$ with zero parts so that it has precisely $n$ parts.

\begin{thm}
\label{schur-matching} 
For a partition $\lambda=(\lambda_1,\ldots,\lambda_n)$, we have 
\begin{equation}
s_{\lambda}(x_1,\ldots,x_n) = 
\m(\T_{n,\lambda_1}^{\lambda_n+1,\lambda_{n-1}+2,\ldots,\lambda_{1}+n}(x_1,\ldots,x_n)).
\end{equation}
\end{thm}

This relation between semistandard tableaux and matchings of the trapezoidal honeycomb graph is not new. 
We illustrate it now with the help of an example where $n=6$. We consider the following semistandard tableau $T$ of shape $(8,6,4,4,2,0)$.
\begin{equation}
\label{eg-ssyt}
\begin{ytableau}
1 & 1 & 1 & 2 & 2 & 5 & 5 & 6 \\
2 & 2 & 3 & 3 & 5 & 6   \\
3 & 4 & 4 & 4    \\
5 & 5 & 6 & 6   \\
6 & 6 
\end{ytableau}
\end{equation}
By using a standard procedure \cite[p.\ 313ff.]{Sta99}, we transform the semistandard tableau into a Gelfand-Tsetlin pattern with $n$ rows as follows: Row $i$ of the Gelfand-Tsetlin pattern is essentially the shape of the entries less than or equal to $i$ in $T$, written in reverse order and filled up with zeros if necessary so that it has length $i$. For the semistandard tableau in \eqref{eg-ssyt}, we obtain the following pattern. 
\begin{equation}
\begin{array}{ccccccccccc}
&&&&&3&&&&& \\
&&&&2&&5&&&& \\
&&&1&&4&&5&&& \\
&&0&&4&&4&&5&& \\
&0&&2&&4&&5&&7& \\ 
0&&2&&4&&4&&6&&8
\end{array}
\end{equation}
The weight of a Gelfand-Tsetlin pattern with $n$ rows is
$\prod_{i=1}^n x_i^{r_i - r_{i-1}}$, where $r_i$ is the sum of the entries in the $i$-th row and $r_0 = 0$.
Now we add $i$ to the $i$-th $\nearrow$-diagonal, where we count the diagonals from left to right.
\begin{equation}
\begin{array}{ccccccccccc}
&&&&&4&&&&& \\
&&&&3&&7&&&& \\
&&&2&&6&&8&&& \\
&&1&&6&&7&&9&& \\
&1&&4&&7&&9&&12& \\ 
1&&4&&7&&8&&11&&14
\end{array}
\end{equation}
We translate this pattern into a matching of $\T_{6,8}^{1,4,7,8,11,14}$; see also \cite[Proposition~2.1]{cohn98}.  First note that the positions of the vertical edges added at the bottom of the graph are just $\lambda_6+1,\lambda_5+2,\ldots,\lambda_1+6$ and that is also the bottom row of the pattern. In general, row $i$ of the pattern lists the positions of the vertical matching edges in row $i$ of $\T_{6,8}(1,4,7,8,11,14)$, counting from the left starting with $1$. All other matching edges are forced then. 
The matching corresponding to our running example in \eqref{eg-ssyt} is given in Figure~\ref{T68_matching}. It is also not difficult to see that the bijection is weight-preserving. (See the proof of Lemma~\ref{prop:even-ortho-nonneg-pattern-character} for a case where we show in detail that a similar bijection is weight-preserving.)

\begin{figure}
\scalebox{0.35}{\includegraphics{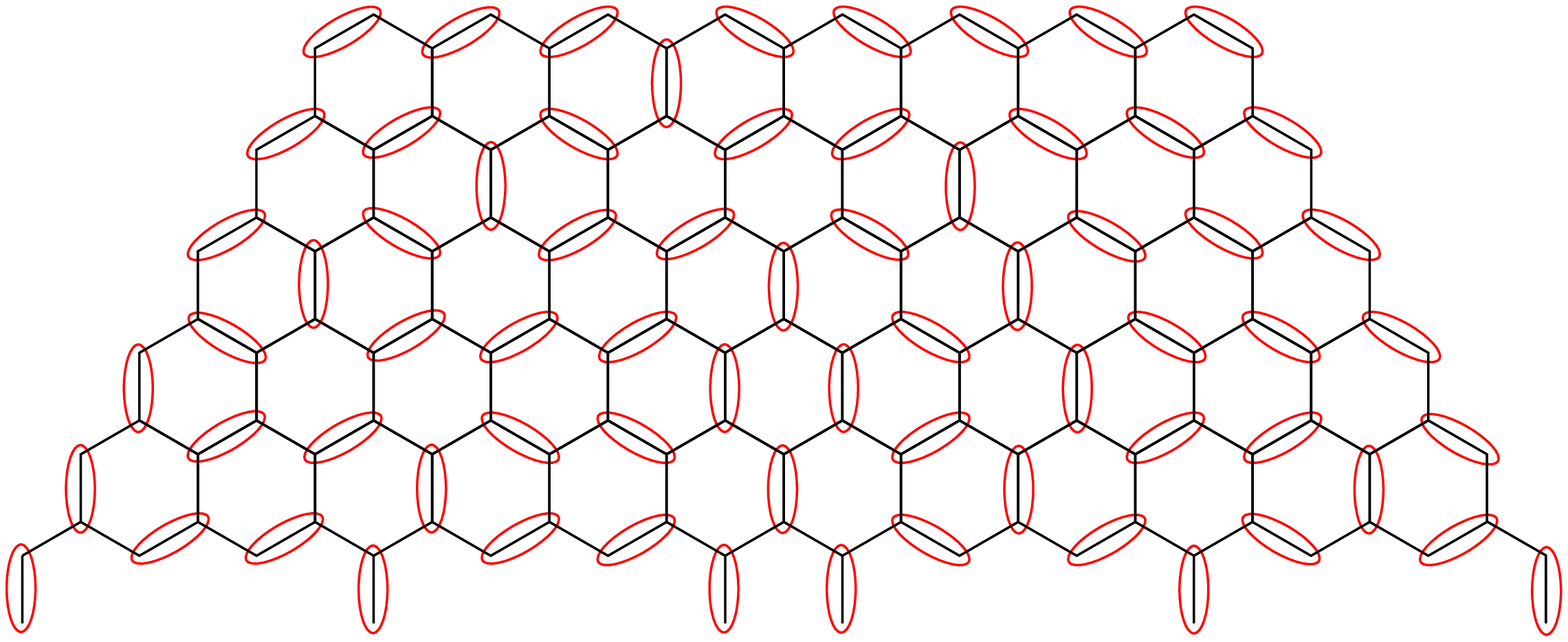}}
\caption{\label{T68_matching} The matching corresponding to the semistandard tableau in \eqref{eg-ssyt}.}
\end{figure}

\subsection{The symplectic group}

As for the other classical groups dealt with in this paper, we rely on variants of Gelfand-Tsetlin patterns~\cite{Kin76,KinElS84,Pro94} . Our patterns are half-turn rotated versions of his, but the labelling of the rows and the inequalities are exactly the same.

To define Gelfand-Tsetlin patterns for the symplectic groups as well as the orthogonal groups, we need the notion of a \emph{half pattern}. 

\begin{defi}[Half patterns]
\label{def:nhalf}
Let $n$ be a positive integer. An {\em $n$-half (Gelfand-Tsetlin) pattern} $P$ is an array of $n$ rows of integers or half-integers of lengths $1,1,2,2,\dots, \lceil n /2 \rceil $ aligned as follows (for $n=6$)
\begin{equation}
\begin{array}{ccccccc}
P_{1,1} \\[0.2cm]
& P_{2,1} \\[0.2cm]
P_{3,1} & & P_{3,2} \\[0.2cm]
& P_{4,1} &  & P_{4,2} \\[0.2cm]
P_{5,1} &  & P_{5,2} &  & P_{5,3} \\[0.2cm]
& P_{6,1} &  & P_{6,2} &  & P_{6,3} 
\end{array},
\end{equation}
such that the entries are weakly increasing along $\nearrow$-diagonals and  
$\searrow$-diagonals.
The first entries in the odd rows are called {\em odd starters}.
\end{defi}

We now define the patterns underlying symplectic characters and their weights.

\begin{defi}[Symplectic patterns]
\label{def:nsympl}
Let $n$ be a positive integer. A {\em $(2n)$-symplectic (Gelfand-Tsetlin) pattern} $P=(P_{i,j})$ is a $(2n)$-half pattern whose entries are all non-negative integers.
The {\em weight} $\ws P$ of a $(2n)$-symplectic pattern $P$ is given by
\begin{equation}
\ws P = \prod_{i=1}^n x_i^{r_{2i}-2r_{2i-1}+r_{2i-2}},
\end{equation}
where $r_i = \sum_j P_{i,j}$ is the sum of entries in row $i$ for $1 \leq i \leq  2n$ and $r_0=0$.
\end{defi}

For a partition $\lambda=(\lambda_1,\ldots,\lambda_n)$, denote the set of all $(2n)$-symplectic patterns with bottom row $\lambda$ in increasing order as $\SP_\lambda$. 
A combinatorial interpretation of symplectic characters in the form of a generating function is provided next.

\begin{thm}[{\cite[Theorem 4.2]{Pro94}}]
\label{proc-symb}
Let $\lambda$ be a partition with $n$ parts. Then 
\begin{equation}
\sp_\lambda(x_1,\ldots,x_n) = \sum_{P \in \SP_\lambda} \ws P.
\end{equation}
\end{thm}

\begin{example}
Let $n=2$ and $\lambda = (1,0)$. Then the four $4$-symplectic patterns contributing to $\sp_{(1,0)}(x_1,x_2)$ and their weights are as follows.
\begin{equation}
\begin{array}{cccc}
\begin{array}{cccc}
1 \\
& 1 \\
0 && 1 \\
& 0 && 1
\end{array}
&
\begin{array}{cccc}
0 \\
& 1 \\
0 && 1 \\
& 0 && 1
\end{array}
&
\begin{array}{cccc}
0 \\
& 0 \\
0 && 1 \\
& 0 && 1
\end{array}
&
\begin{array}{cccc}
0 \\
& 0 \\
0 && 0 \\
& 0 && 1
\end{array} \\
\hline
\x_1 & x_1 & \x_2 & x_2 
\end{array}
\end{equation}
\end{example}

Analogous to the case of Schur polynomials, we express the symplectic characters as a matching generating function of certain weighted graphs. Let $\htm{}{2n,k}$ denote the half-trapezoidal honeycomb graph consisting of $2n-1$ left-justified rows of consecutive hexagons of lengths $k,k,k+1,k+1,\dots,k+n-2,k+n-2,k+n-1$, with one edge added incident with the bottom vertex of the rightmost vertical edge; see Figure~\ref{fig:eg-htm-graph} for an example. Note that the odd and even rows are vertically aligned separately, and we fix the convention that the odd rows are shifted by half a hexagon to the left of the even rows. Just like $\T_{n,k}$, the bipartite graph $\htm{}{2n,k}$ does not have a perfect matching as it has $n$ more vertices in one vertex class than in the other. We attach vertical edges to $n$ of the $n+k$ bottommost vertices at positions $p_1< \dots < p_n$ (numbered $1, \dots, n+k$ from left to right) to form the graph $\htm{\mathbf{p}}{2n,k}$. 
We also consider an edge-weighted version: The graph $\htm{\mathbf{p}}{2n,k}(x_1,\dots,x_{2n})$ is formed by weighting each edge of type \senw in row $i$ by $x_i$. All other edges are weighted 1.
\begin{figure}
\begin{center}
\scalebox{1.2}{\includegraphics{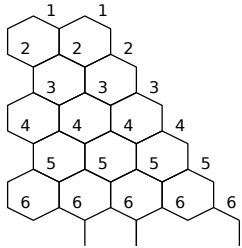}}
\caption{The half-trapezoidal honeycomb graph 
$\htm{2,3,5}{6,2}(x_1,\dots,x_6)$.}
\label{fig:eg-htm-graph}
\end{center}
\end{figure}

\begin{thm}
\label{prop:symplectic-pattern-character}
For a partition $\lambda = (\lambda_1,\dots,\lambda_n)$, we have
\begin{equation}
\sp_\lambda(x_1,\ldots,x_n)
 = M(\htm{\lambda_n+1,\lambda_{n-1}+2,\dots,\lambda_1+n}
 {2n,\lambda_1}
 (x_1,\x_1, \dots,x_{n},\x_n)).
\end{equation}
\end{thm}

In Figure~\ref{fig:eg-htm-matching}, we illustrate Theorem~\ref{prop:symplectic-pattern-character} with an example where the underlying graph is the one in Figure~\ref{fig:eg-htm-graph}.

\begin{figure}[htbp!]
\begin{center}
\begin{tabular}{ccc}
\includegraphics[scale=1]{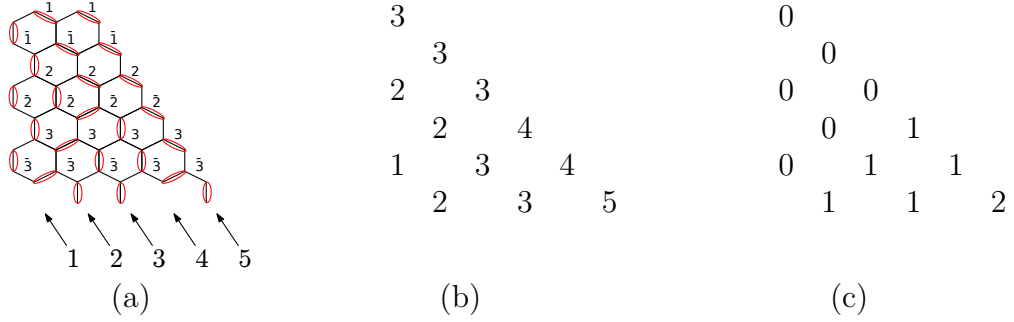} 
&
\hspace*{1cm}
\raisebox{2.0cm}{
\begin{tabular}{cccccc}
3 \\
& 3 \\
2 && 3 \\
& 2 && 4 \\
1 && 3 && 4 \\
& 2 && 3 && 5
\end{tabular}
}
&
\hspace*{1cm}
\raisebox{2.0cm}{
\begin{tabular}{cccccc}
0 \\
& 0 \\
0 && 0 \\
& 0 && 1 \\
0 && 1 && 1 \\
& 1 && 1 && 2
\end{tabular}
} 
\\
(a) & (b) & (c) 
\end{tabular}
\caption{
(a) A matching of the half-trapezoidal honeycomb graph 
$\htm{2,3,5}{6,2}(x_1,x_2,x_3,\x_1,\x_2,\x_3)$ corresponding to the partition $\lambda = (2,1,1)$ with weight 
$x_2 x_3$. The numbering of the vertical edges is indicated below.
(b) The positions of the vertical edges in rows $1$ through $6$ arranged in the shape of a $6$-half pattern.
(c) A $6$-symplectic pattern with bottom row $\lambda$ having the same weight as the matching on the left.}
\label{fig:eg-htm-matching}
\end{center}
\end{figure}

\begin{proof}
Similar to the case of Schur polynomials, we construct a weight-preserving bijection between the sets corresponding to the two sides of the equation.

First of all, note that a matching in 
$\htm{p_1,\dots,p_n}{2n,k}$
is completely determined by the positions of the vertical edges in each row. Further, the number of vertical edges in the $i$-th row can be seen to be $\lceil \frac{i}{2} \rceil$. Let the positions of the vertical edges in row $i$ be labelled starting from the left with $n+1-\lceil \frac{i}{2} \rceil$.

Arranging the positions of these vertical edges in the form of a $(2n)$-half pattern, we see that the $\nearrow$-diagonals are strictly increasing, and the $\searrow$-diagonals are weakly increasing. Therefore, decreasing the $j$-th 
$\searrow$-diagonal (counted from the left, starting with $1$) by $j$, we obtain a $(2n)$-symplectic pattern according to Definition~\ref{def:nsympl}. This construction is illustrated in Figure~\ref{fig:eg-htm-matching}.

The proof that this map is weight-preserving is very similar to that for the general linear group. (We give more details on this in a similar situation in the proof of Lemma~\ref{prop:even-ortho-nonneg-pattern-character}.)
\end{proof}

\subsection{The even orthogonal group}

We now define Gelfand-Tsetlin patterns for the even orthogonal groups. 
For an $n$-half pattern $P=(P_{i,j})$ as described in Definition~\ref{def:nhalf},
the {\em absolute row sum} of the $i$-th row is given by $r_i^+ = \sum_j |P_{i,j}|$.
Let $\sgn(x) = \begin{cases} 1 & x \geq 0, \\ -1 & x<0. \end{cases}$

\begin{defi}[Orthogonal patterns]
\label{defi_evenpatterns}
Let $n$ be a positive integer. A {\em $(2n-1)$-orthogonal (Gelfand-Tsetlin) pattern} $P$ is a $(2n-1)$-half pattern whose entries are either all integers or all half-integers, and which satisfy the following conditions:
\begin{itemize}
\item All entries except the odd starters are non-negative.
\item The odd starters satisfy 
$ |P_{2i-1,1}| \leq \min(P_{2i-2,1},P_{2i,1})$ (with $P_{0,1},P_{2n,1}=+\infty$).
\end{itemize}
The {\em weight} $\woe P$ of a $(2n-1)$-orthogonal pattern $P$ is given by
\begin{equation}
\label{wt-even-orthog}
\woe P = \prod_{i=1}^{n} x_i^{\sgn(P_{2i-1,1}) \sgn(P_{2i-3,1}) (r_{2i-1}^+ - 2r_{2i-2}^+ + r_{2i-3}^+)},
\end{equation}
where we set $r^+_0 = r^+_{-1} = 0$ and $P_{-1,1} = 0$.
\end{defi}

For an integer partition or a half-integer partition $\lambda = 
(\lambda_1,\dots,\lambda_n)$, we set 
$\lambda^- = (\lambda_1,\dots,\lambda_{n-1},\allowbreak -\lambda_n)$.
Denote the set of all $(2n-1)$-orthogonal patterns with bottom row $\lambda$ or $\lambda^-$ in increasing order as $\OP_\lambda$ or $\OP_{\lambda^{-}}$, respectively. 

\begin{thm}[{\cite[First part of Theorem 7.3]{Pro94}}]
\label{even}
Let $\lambda=(\lambda_1,\ldots,\lambda_n)$ be either an integer partition or a half-integer partition. Then 
\begin{equation}
\oe_\lambda(x_1,\ldots,x_n) 
= \sum_{P \in \OP_\lambda \cup \OP_{\lambda^-}} \woe P.
\end{equation}
\end{thm}

\begin{example}
Let $n=2$ and $\lambda = (1,1)$. Then the six $3$-orthogonal patterns contributing to $\oe_{(1,1)}(x_1,x_2)$ and their weights are as follows.
\begin{equation}
\begin{array}{cccccc}
\begin{array}{ccc}
-1 \\
& 1 \\
1 && 1
\end{array}
&
\begin{array}{ccc}
0 \\
& 1 \\
1 && 1
\end{array}
&
\begin{array}{ccc}
1 \\
& 1 \\
1 && 1
\end{array}
&
\begin{array}{ccc}
-1 \\
& 1 \\
-1 && 1 
\end{array}
&
\begin{array}{ccc}
0 \\
& 1 \\
-1 && 1 
\end{array}
&
\begin{array}{ccc}
1 \\
& 1 \\
-1 && 1 
\end{array} \\
\hline
\x_1 \x_2 & 1 & x_1 x_2 &
\x_1 x_2 & 1 & x_1 \x_2
\end{array}
\end{equation}
\end{example}

Again we aim to relate the even orthogonal characters to the matching generating functions of certain honeycomb graphs. This relation is not as straightforward as in the case of symplectic characters in Theorem~\ref{prop:symplectic-pattern-character}. 
Let $\htp{}{2n,k}$ be a half-trapezoidal honeycomb graph consisting of $2n-1$ right-justified rows of consecutive hexagons of length $k, k+1, k+1, \dots, k+n-1,k+n-1$ with three extra edges (see Figure~\ref{fig:eg-htp-graph} for an example): one at the end of the top zig-zag row, and two at the beginning and end  of the bottom zig-zag row.  Again, the odd and even rows are aligned separately, and we fix the even rows to be aligned half a hexagon to the right of the odd rows. 

We attach $n$ vertical edges to vertices in the bottom row at positions $p_1 < \dots < p_n$ numbered $0,1, \dots, n+k$ from \emph{right to left} to form the graph $\htp{\mathbf{p}}{2n,k}$. 
In the edge-weighted variant $\htp{\mathbf{p}}{2n,k}(x_1,\dots,x_n)$, each edge of type \swne in row $i$ is weighted $x_i$ and rightmost vertical edges are weighted by $\frac{1}{2}$. The version where the rightmost vertical edge are weighted by $1$ is denoted by $\hhtp{\mathbf{p}}{2n,k}(x_1,\dots,x_n)$.\footnote{Note that combining $\htp{p_1,\dots,p_n}{2n,k}$ with $\htm{q_1,\dots,q_n}{2n,k}$ by adding one edge to each of the $2n$ leftmost vertices of the latter graph, we obtain the trapezoidal graph $\T_{2n,2k+1}^{n+k+1-p_n,n+k+1-p_{n-1},\ldots,n+k+1-p_1,n+k+1+q_1,\dots,n+k+1+q_n}$.}

\begin{figure}[htbp!]
\begin{center}
\includegraphics[scale=1.2]{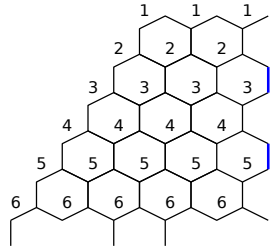}
\caption{The half-trapezoidal honeycomb graph 
$\htp{2,3,5}{6,2}(x_1,\dots,x_6)$. The blue edges on the right are weighted $\frac{1}{2}$.}
\label{fig:eg-htp-graph}
\end{center}
\end{figure}

To prove the main result of this section, we need a preliminary identity involving a subset of orthogonal patterns. We say that an orthogonal pattern is {\em non-negative} if all its entries are non-negative. 

\begin{lem}
\label{prop:even-ortho-nonneg-pattern-character}
For a partition $\lambda = (\lambda_1,\dots,\lambda_n)$, we have
\begin{equation}
 \sum_{\substack{P \in \OP_\lambda \\ \text{$P$ non-negative}}} \woe P  = 
M(\hhtp{\lambda_n,\lambda_{n-1}+1,\dots,\lambda_1+n-1}
 {2n,\lambda_1-1}
 (x_1,\x_1,\dots,x_{n},\x_n)).
\end{equation}
\end{lem}

We illustrate Lemma~\ref{prop:even-ortho-nonneg-pattern-character} with an example in Figure~\ref{fig:eg-htp-matching} where the underlying graph is the one given in Figure~\ref{fig:eg-htp-graph}.

\begin{figure}[htbp!]
\begin{center}
\begin{tabular}{ccc}
\includegraphics[scale=1]{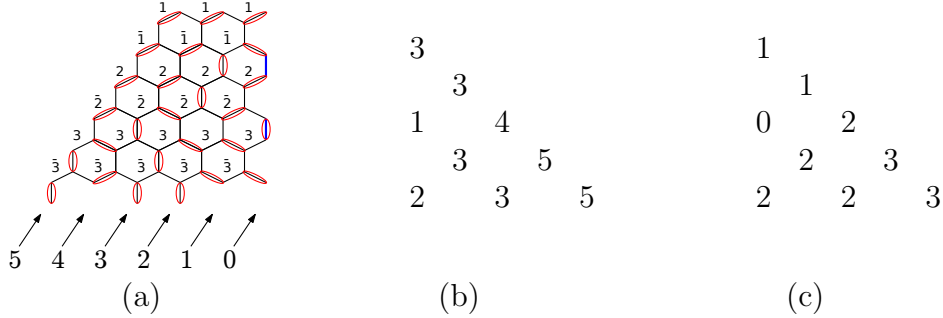} 
&
\hspace*{1cm}
\raisebox{1.8cm}{
\begin{tabular}{cccccc}
3 \\
& 3 \\
1 && 4 \\
& 3 && 5 \\
2 && 3 && 5 
\end{tabular}
}
&
\hspace*{1cm}
\raisebox{1.8cm}{
\begin{tabular}{cccccc}
1 \\
& 1 \\
0 && 2 \\
& 2 && 3 \\
2 && 2 && 3 
\end{tabular}
} \\
(a) & (b) & (c) 
\end{tabular}
\caption{
(a) A matching of the half-trapezoidal honeycomb graph 
$\htp{2,3,5}{6,2}(x_1,\x_1,x_2,\x_2,x_3,\x_3)$ corresponding to the partition $\lambda = (3,2,2)$ with weight 
$(x_1 x_2 \x_3)/2$. The positions of the vertical edges are marked below. 
(b) The positions of the vertical edges in rows $2$ through $6$ arranged in the shape of a $5$-half pattern.
(c) The corresponding $5$-orthogonal pattern with bottom row $\lambda$ is displayed which is obtained by subtracting $i$ from the $i$-th $\searrow$-diagonal, where diagonals are counted from the left starting with $0$.}
\label{fig:eg-htp-matching}
\end{center}
\end{figure}

\begin{proof}
We construct a weight-preserving bijection between $(2n-1)$-orthogonal patterns with all odd starters being non-negative and $\mathcal{M}(\hhtp{\lambda_n,\lambda_{n-1}+1,\dots,\lambda_1+n-1} {2n,\lambda_1-1} (x_1,\x_1,\dots,x_{n},\x_n))$. This proceeds similar to the proof of Theorem~\ref{prop:symplectic-pattern-character}.

As usual, each matching is uniquely determined by the positions of the vertical edges in each row, with there being $ \lfloor \frac{i}{2} \rfloor$ matching edges in the $i$-th row. (Note that the edges in the top zig-zag row are forced because of the jutting edge ``/'' on the right, and, therefore, there are no vertical matching edges in the top row.) Label the positions of the vertical edges starting from the left in each row with $n+k$ and decreasing to the right, and construct an array that has the form of a $(2n-1)$-half pattern by listing the positions of the vertical edges in each row in increasing order. 

We then see that the $\nearrow$-diagonals are strictly increasing, while the $\searrow$-diagonals are weakly increasing. Subtract $i$ from the $i$-th $\searrow$-diagonal, counting from the left starting with $0$. The odd starters are non-negative by construction. 

We now show that this bijection is weight-preserving. 
The formula for the weight of the pattern in \eqref{wt-even-orthog} simplifies since $r_i^+ = r_i$ for each $i$ and the signs are all $1$. In the graphical model, let the positions of the vertical matching edges in the $i$-th row in increasing order (i.e.\ right to left) be $T_{i,1}, \dots, T_{i,\floor{i/2}}$ for $i=1,2,\ldots,2n$ (see the middle column of Figure~\ref{fig:eg-htp-matching}). The corresponding row of the orthogonal pattern $P$ is given by $P_{i-1,1}, \dots, P_{i-1,\floor{i/2}}$, where 
\begin{equation}
\label{PT-relation}
P_{i-1,j} = T_{i,j} - (n-\floor{i/2} + j-1). 
\end{equation}

To compare the exponents of $x_i$, we need to consider the matching edges  of type \swne in rows $2i-1$ and $2i$ of zig-zags. To this end, we need consider the vertically matching edges in rows $2i-2, 2i-1$ and $2i$. These are given by the tuples $(T_{2i-2,1},\dots,T_{2i-2,i-1})$, 
$(T_{2i-1,1},\dots,T_{2i-1,i-1})$ and
$(T_{2i,1},\dots,T_{2i,i})$. The number of matching edges of type \swne in row $2i-1$  is 
\begin{multline}
(T_{2i-2,1} - (n-i+1)) +  (T_{2i-2,2} - T_{2i-1,1} - 1) 
+ \cdots + (T_{2i-2,i-1} - T_{2i-1,i-2} - 1) 
+ (n + \lambda_1 - 1 - T_{2i-1,i-1}) 
 \\ = \sum_{j=1}^{i-1} (T_{2i-2,j} - T_{2i-1,j}) + \lambda_1.
\end{multline}
Similarly, the number of matching edges of type ``/'' in row $2i$ is
\begin{multline}
 (T_{2i-1,1} - T_{2i,1} - 1)
+ \cdots + (T_{2i-1,i-1} - T_{2i,i-1} - 1) 
+ (n + \lambda_1 - 1 - T_{2i,i}) \\
= \sum_{j=1}^{i-1} T_{2i-1,j} - \sum_{j=1}^{i} T_{2i,j}
- i + n + \lambda_1.
\end{multline}
Therefore, the total exponent of $x_i$ is the difference of these expressions, which is
\begin{equation}
\sum_{j=1}^{i-1} T_{2i-2,j} -2 \sum_{j=1}^{i-1} T_{2i-1,j} + \sum_{j=1}^{i} T_{2i,j} - n + i.
\end{equation}
Using \eqref{PT-relation}, this can be shown to be equal to 
$r_{2i-1} - 2r_{2i-2} + r_{2i-3}$, as desired.
\end{proof}

We need to consider the following operation on orthogonal patterns. 
This is a variation of the Bender-Knuth involution~\cite{BenKnu72}.

\begin{prop}
\label{prop:OP-involution}
Let $\lambda = (\lambda_1,\dots,\lambda_n)$ be a partition.
Define the map $J_i$ on orthogonal patterns $\OP_\lambda$ 
for $2 \leq i \leq n$ as follows. For $P \in \OP_\lambda$, $J_i(P)$ leaves all the rows of $P$, except for row $2i-2$ unchanged, and
\begin{equation}
(J_i(P))_{2i-2,j} = \max(|P_{2i-3,j}|,|P_{2i-1,j}|) + \min(P_{2i-3,j+1},P_{2i-1,j+1}) - P_{2i-2,j}, \quad 1 \leq j \leq i-1,
\end{equation}
where we fix $P_{2i-3,i} = +\infty$.
Then the following properties hold for $J_i$.
\begin{itemize}
\item $J_i(P) \in \OP_\lambda$.
\item $J_i$ is an involution.
\item If either of the odd starters in rows $2i-3$ or $2i-1$ is $0$,
then the  exponent of $x_i$ in $\woe{J_i(P)}$ is the negative of the exponent of $x_i$ in $\woe{P}$.
\end{itemize}
\end{prop}

We will only apply the operation $J_i$ to orthogonal patterns such that either of the odd starters in rows $2i-3$ or $2i-1$ is $0$.  The proof of Proposition~\ref{prop:OP-involution} is a routine calculation and left to the interested reader.

We denote by $\mathcal{I}_n$ the set of permutations of 
$
\{x_1,\bar x_1, x_2, \bar x_2, \ldots, x_n, \bar x_n\}
$
that are generated by the transpositions 
$
(x_1, \bar x_1), (x_2, \bar x_2), \ldots, (x_n, \bar x_n).
$
This group is isomorphic to $\mathbb{Z}_2^n$. For $\sigma \in \mathcal{I}_n$, 
$\sigma \, \htp{p_1,\dots,p_n}{2n,k} \allowbreak (x_1,\bar x_1,\dots,x_n,\bar x_n)$ denotes the edge weighted graph $\htp{p_1,\dots,p_n}{2n,k}$ whose row-weights are permuted according to $\sigma$.

\begin{thm}
\label{prop:even-ortho-pattern-character}
For a partition $\lambda = (\lambda_1,\dots,\lambda_n)$, we have
\begin{equation}
\oe_\lambda(x_1,\ldots,x_n)
 = \sum_{\sigma \in \mathcal{I}_n} 
 M(\sigma \, \htp{\lambda_n,\lambda_{n-1}+1,\dots,\lambda_1+n-1}
 {2n,\lambda_1-1}
 (x_1,\x_1,\dots,x_{n},\x_n)).
\end{equation}
\end{thm}

\begin{proof}
We use Theorem~\ref{even} to interpret the left-hand side.  To be more precise, we need to refine this interpretation in the following sense: each orthogonal pattern $P$ appearing on the left-hand side that has precisely $k$ odd starters equal to $0$ is replaced by $2^k$ copies of $P$, the weight of each copy being $\frac{1}{2^k} \woe P$, and we accompany each copy with a (different) $\{0,1\}$-sequence of length $k$ to remedy this modification. 

Fix a pair $(m,\sigma)$, where $m$ is a matching of $\htp{\lambda_n,\lambda_{n-1}+1,\dots,\lambda_1+n-1} {2n,\lambda_1-1}$ and $\sigma \in \mathcal{I}_n$. We give a weight-preserving bijection between such pairs (where the weight is the weight of the matching in $\sigma \, \htp{\lambda_n,\lambda_{n-1}+1,\dots,\lambda_1+n-1}
 {2n,\lambda_1-1}$) and objects from the left-hand side as described in the previous paragraph.

First, let $P$ be the orthogonal non-negative pattern corresponding to $m$ by Lemma~\ref{prop:even-ortho-nonneg-pattern-character}. Now we change inductively, for each $i \in \{1,2,\ldots,n\}$, the odd starter $P_{2i-1}$ (by changing the sign) or row $2i-2$ of $P$ (by applying $J_i$) according to $\sigma$ such that the map is weight-preserving. Here the crucial observation is that the weight of a matching in 
\begin{equation} (x_i,\bar x_{i}) \sigma \, \htp{\lambda_n,\lambda_{n-1}+1,\dots,\lambda_1+n-1}
 {2n,\lambda_1-1} (x_1,\x_1,\dots,x_{n},\x_n)\end{equation} is obtained from its weight in 
 \begin{equation} \sigma \, \htp{\lambda_n,\lambda_{n-1}+1,\dots,\lambda_1+n-1}
 {2n,\lambda_1-1} (x_1,\x_1,\dots,x_{n},\x_n)\end{equation}
 by replacing $x_i$ with $\bar x_i$. 
 
Let $\sigma = (x_1,\bar x_1)^{\epsilon_1} \dots (x_n,\bar x_n)^{\epsilon_n}$ for appropriately chosen $\epsilon_i \in \{0,1\}$. 
For $i=1$, we do the following: If $P_{1,1} > 0$, we change the sign of $P_{1,1}$ iff 
$\epsilon_1=1$; if $P_{1,1}=0$, we leave $P$ unchanged and the corresponding position in the accompanying $\{0,1\}$-sequence is 
$\epsilon_1$. 

 Now suppose we have reached $i>1$. First assume $P_{2i-1,1} > 0$. If $\epsilon_i=0$, we give  $P_{2i-1,1}$  the appropriate sign so that it has the same sign 
 as $P_{2i-3,1}$, otherwise so that it has the opposite sign. 
If, however, 
$P_{2i-1,1}=0$ we apply $J_i$ to $P$ if and only if either $\epsilon_i=1$ and $P_{2i-3,1} \ge 0$, or $\epsilon_i=0$ and $P_{2i-3,1}<0$. If we have applied $J_i$, we record $1$ in the appropriate position of the accompanying $\{0,1\}$-sequence and $0$ otherwise.

To give an example, we consider the matching in Figure~\ref{fig:eg-htp-matching} and  
$\sigma=(x_1, \x_1) (x_3,\x_3)$ to it. The weight is then $(\x_1 x_2 x_3)/2$. Applying the procedure just described, we obtain the following $5$-orthogonal pattern with the same weight
\begin{equation}
\begin{tabular}{cccccc}
$-1$ \\
& $2$ \\
$0$ && $2$ \\
& $2$ && $2$ \\
$2$ && $2$ && $3$ 
\end{tabular}
\end{equation}
and the accompanying $\{0,1\}$-sequence of length $1$ is $(1)$. 
\end{proof}

We now consider the case when the bottom row of the $(2n-1)$-orthogonal pattern consists of half-integers. For this purpose, the following 
variation of Lemma~\ref{prop:even-ortho-nonneg-pattern-character} is helpful.

\begin{lem}
\label{lem:even-ortho-nonneg-pattern-character-half}
For a half-integer partition $\lambda = (\lambda_1,\dots,\lambda_n)$, we have 
\begin{equation}
 \sum_{\substack{P \in \OP_\lambda \\ \text{$P$ non-negative}}} \woe P  = 
\prod_{i=1}^{n} x_i^{1/2} M(\hhtp{\left(\lambda_n - \frac{1}{2}\right),\left(\lambda_{n-1}-\frac{1}{2}\right)+1,\dots,\left(\lambda_1-\frac{1}{2}\right)+n-1}
 {2n,\left(\lambda_1-\frac{1}{2}\right)-1}
 (x_1,\x_1,\dots,x_{n},\x_n)).
\end{equation}
\end{lem}

\begin{proof} This follows from 
\begin{equation}
\sum_{\substack{P \in \OP_\lambda \\ \text{$P$ non-negative}}} \woe P =
\prod_{i=1}^{n} x_i^{1/2} \sum_{\substack{P \in \OP_{\lambda-\frac{1}{2}} \\ \text{$P$ non-negative}}} \woe P
\end{equation}
and Lemma~\ref{prop:even-ortho-nonneg-pattern-character}.
\end{proof}

We obtain the following combinatorial interpretation for even orthogonal characters for half-integer partitions.

\begin{thm}
\label{prop:even-ortho-pattern-character-half}
For a half-integer partition $\lambda = (\lambda_1,\dots,\lambda_n)$, we have
\begin{equation}
\oe_\lambda(x_1,\ldots,x_n)
 = \sum_{\sigma \in \mathcal{I}_n} \left[ \sigma \prod_{i=1}^{n} x_i^{1/2} \right] \cdot
 M(\sigma \, \hhtp{\left(\lambda_n-\frac{1}{2} \right),\left(\lambda_{n-1}-\frac{1}{2} \right)+1,\dots,\left(\lambda_1-\frac{1}{2} \right)+n-1}
 {2n,\left(\lambda_1-\frac{1}{2} \right)-1}
 (x_1,\x_1,\dots,x_{n},\x_n)),
\end{equation}
where $\sigma \prod\limits_{i=1}^{n} x_i^{1/2}$ is obtained from $\prod\limits_{i=1}^{n} x_i^{1/2}$ by replacing $x_i$ with $\bar x_i$ iff $(x_i, \bar x_i)$ appears in $\sigma$.
\end{thm}

\begin{proof} The proof is similar to the proof of Theorem~\ref{prop:even-ortho-pattern-character}, but considerably simpler because there are no starters equal to $0$. In particular, the maps $J_i$ defined in Proposition~\ref{prop:OP-involution} are not needed.
\end{proof}

\subsection{The odd orthogonal groups}

\begin{defi}[Split orthogonal patterns]
\label{defi_splitpatterns}
Let $n$ be a positive integer. A {\em $(2n)$-split orthogonal (Gelfand-Tsetlin) pattern} $P$ is a $(2n)$-half pattern in which the entries, except for the odd starters, are either all non-negative integers or all non-negative half-integers; each starter is independently either a non-negative integer or a non-negative half-integer.
The {\em weight} $\woo P$ of a $(2n)$-split pattern $P$ is given by
\begin{equation}
\woo P = \prod_{i=1}^n x_i^{r_{2i}-2r_{2i-1}+r_{2i-2}},
\end{equation}
where we set $r_0 = 0$.
\end{defi}

Denote the set of all $(2n)$-split orthogonal patterns with bottom row $\lambda$ in increasing order as $\SOP_\lambda$.

\begin{thm}[{\cite[First part of Theorem 7.1]{Pro94}}]
Let $\lambda$ be a partition with $n$ parts. 
Then 
\begin{equation}
\oo_\lambda(x_1,\ldots,x_n) = \sum_{P \in \SOP_\lambda} \woo P.
\end{equation}
\end{thm}

\begin{example}
Let $n=2$ and $\lambda = (1,0)$. Then the five $4$-split orthogonal patterns contributing to $\oo_{(1,0)}(x_1,x_2)$ and their weights are
\begin{equation}
\begin{array}{ccccc}
\begin{array}{cccc}
1 \\
& 1 \\
0 && 1 \\
& 0 && 1
\end{array}
&
\begin{array}{cccc}
1/2 \\
& 1 \\
0 && 1 \\
& 0 && 1
\end{array}
&
\begin{array}{cccc}
0 \\
& 1 \\
0 && 1 \\
& 0 && 1
\end{array}
&
\begin{array}{cccc}
0 \\
& 0 \\
0 && 1 \\
& 0 && 1
\end{array}
&
\begin{array}{cccc}
0 \\
& 0 \\
0 && 0 \\
& 0 && 1
\end{array} \\
\hline
\x_1 & 1 & x_1 & \x_2 & x_2 
\end{array}
\end{equation}
\end{example}

Let $\hhtm{p_1,\dots,p_n}{2n,k}(x_1,\dots,x_{2n})$ denote the weighted graph obtained from $\htm{p_1,\dots,p_n}{2n,k}(x_1,\dots,x_{2n})$ by adding $1$ to the weight of the first edge in the even rows of zigzags (which is always an edge of type \senw), so that this edge in row $2i$ then has weight $ x_{2 i} + 1$. Similarly to Theorem~\ref{prop:symplectic-pattern-character}, one can show the following.

\begin{thm}
\label{oddorth} For a partition $\lambda=(\lambda_1,\ldots,\lambda_n)$, we 
have 
\begin{equation}\oo_\lambda(x_1,\ldots,x_n) = \m(\hhtm{\lambda_n+1,\lambda_{n-1}+2,\ldots,\lambda_1+n}{2n,\lambda_1}(x_1,\bar x_1,\ldots,x_n,\bar x_n)).
\end{equation}
\end{thm}

\begin{proof} For a $(2n)$-split orthogonal pattern, we add $\frac{1}{2}$ to the odd starters that are half-integers to obtain a $(2n)$-symplectic pattern. Using Theorem~\ref{prop:symplectic-pattern-character}, we then obtain the corresponding matching of 
\begin{equation}\htm{\lambda_n+1,\lambda_{n-1}+2,\ldots,\lambda_1+n}{2n,\lambda_1}(x_1,\bar x_1,\ldots,x_n,\bar x_n).\end{equation} 
This is clearly a 
$2^r$-to-$1$ mapping, where $r$ is the number of non-zero odd starters. The weight $\bar x_i +1$ of the first edge of type \senw in row 
$2i$ of 
\begin{equation}\hhtm{\lambda_n+1,\lambda_{n-1}+2,\ldots,\lambda_1+n}{2n,\lambda_1}(x_1,\bar x_1,\ldots,x_n,\bar x_n)\end{equation}
remedies this for the following reason. 
This edge is in the matching if and only if the starter in row $2i-1$ is non-zero. The weight of a $(2n)$-split orthogonal pattern where the odd starter in row $2i-1$ is a half integer is obtained from the weight of the pattern with this odd starter rounded up by multiplying with $x_i$.
\end{proof}

\section{Symmetrizing the weights}
\label{sec_symmetrizing}

Now that we have combinatorial realizations of all the quantities in Theorem~\ref{thm:fact1}, we can proceed with our combinatorial proof. The proof consists of two major steps. The purpose of this section is to provide necessary ingredients for the first step. The second step is then the application of Ciucu's 
\emph{factorization theorem for graphs with reflective symmetry}  \cite{Ciu97}. The unweighted graphs $\T_{n,k}^{\mathbf{p}}$ that are relevant for  Theorem~\ref{thm:fact1} are indeed already endowed with such a reflective symmetry. However, the edge weights of 
$\T_{n,k}^{\mathbf{p}}(x_1,\ldots,x_n)$ are not symmetric and Ciucu's theorem is only applicable if they are. In this section, we show how, in the context of the first part of Theorem~\ref{thm:fact1}, it is possible to reduce the problem to graphs where the edge weights are also symmetric. (Later we will see that the underlying procedure is also useful for proving the second part of Theorem~\ref{thm:fact1}.)
For this purpose, it is necessary to introduce the following different edge weights of $\T_{n,k}^{\mathbf{p}}$.

\begin{defi}[$\ST_{n,k,j}^{\mathbf{p}}(x_1,\ldots,x_n)$] 
Let $\ell_j$ denote the vertical line that contains the $j$-th vertex in the top row, counted from the left, of $\T_{n,k}^{p_1,\ldots,p_n}$. To the left of $\ell_j$, the weights of 
$\ST_{n,k,j}^{\mathbf{p}}(x_1,\ldots,x_n)$ coincide with the weights of  $\T_{n,k,j}^{\mathbf{p}}(x_1,\ldots,x_n)$, and to the right of $\ell_j$ the edges of type \senw are assigned the weight $x_i$ in row $i$, while all other edges are assigned the weight $1$. 
\end{defi} 

See Figure~\ref{ST68} for  $\ST_{6,8,4}^{1,4,7,8,11,14}(x_1,\ldots,x_6)$. 

\begin{figure}
\scalebox{0.35}{
\psfrag{1}{\Large$1$}
\psfrag{2}{\Large$2$}
\psfrag{3}{\Large$3$}
\psfrag{4}{\Large$4$}
\psfrag{5}{\Large$5$}
\psfrag{6}{\Large$6$}\includegraphics{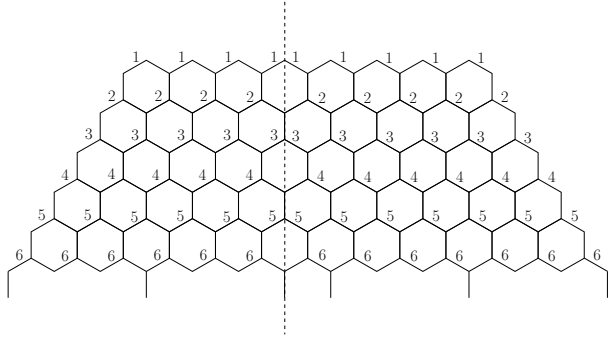}}
\caption{\label{ST68} The graph $\ST_{6,8,4}^{1,4,7,8,11,14}(x_1\ldots,x_6)$.}
\end{figure}

In the following, it is convenient to use a concept which we call \emph{randomized bijection}: Suppose $A$ and $B$ are two finite sets, then a randomized map from $A$ to $B$ is a randomized algorithm that assigns to each element $a \in A$ an element $b \in B$ with some probability $p_{a,b}$ such that
\begin{equation}
\sum_{b \in B} p_{a,b} = 1
\end{equation}
for all $a \in A$. We say that a randomized map is a randomized bijection, if there exists a randomized map from $B$ to $A$ such that the corresponding randomized algorithm sends $b \in B$ to $a \in A$ with probability $p_{a,b}$ (which implies then also $\sum_{a \in A} p_{a,b} = 1$ for all $b \in B$). 
A randomized bijection can only exist if $A$ and $B$ have the same cardinality, as 
\begin{equation}
|A| = \sum_{a \in A} 1 = \sum_{a \in A} \sum_{b \in B} p_{a,b} = \sum_{b \in B} \sum_{a \in A} p_{a,b} = \sum_{b \in B} 1 = |B|.
\end{equation} 
This concept can also be used to show that the generating functions of two sets $A$ and $B$ are the same: if a randomized bijection from $A$ and $B$ is weight-preserving (that is $p_{a,b}=0$ unless $\w(a)=\w(b)$), then it is a randomized bijection between the subset of elements of $A$ that have a prescribed weight $w$ and the subset of elements of $B$ that have the same weight $w$, for every possible weight $w$.

\begin{lem} 
\label{line}
Let $1 \le j \le k$. Then the matching generating function \begin{equation}\m(\T_{2n,k}^{\mathbf{p}}(x_1,\bar x_1,\ldots,x_{n}, \bar x_{n}))\end{equation} is equal to 
\begin{equation}
\frac{1}{2^{n}}  \sum_{\sigma \in \mathcal{I}_n}  \m(\sigma \, \ST_{2n,k,j}^{\mathbf{p}}(x_1,\bar x_1,\ldots,x_n, \bar x_n)).
\end{equation}
\end{lem}

\begin{proof} 
By Theorem~\ref{schur-matching} and the symmetry of the Schur polynomial, 
\begin{equation}
\m(\sigma  \, \T_{2n,k}^{\mathbf{p}}(x_1,\bar x_1,\ldots,x_{n}, \bar x_{n})) =
\m( \T_{2n,k}^{\mathbf{p}}(x_1,\bar x_1,\ldots,x_{n}, \bar x_{n}))
\end{equation}
for any $\sigma \in \mathcal{I}_n$. It follows that 
\begin{equation}
\m( \T_{2n,k}^{\mathbf{p}}(x_1,\bar x_1,\ldots,x_{n}, \bar x_{n})) = \frac{1}{2^n} \sum_{\sigma \in \mathcal{I}_n} \m(\sigma \, \T_{2n,k}^{\mathbf{p}}(x_1,\bar x_1,\ldots,x_{n}, \bar x_{n}))
\end{equation}
and therefore it suffices to show 
\begin{equation}
\sum_{\sigma \in \mathcal{I}_n} \m(\sigma \, \T_{2n,k}^{\mathbf{p}}(x_1,\bar x_1,\ldots,x_{n}, \bar x_{n})) = 
\sum_{\sigma \in \mathcal{I}_n} \m(\sigma \ST_{2n,k,j}^{\mathbf{p}}(x_1,\bar x_1,\ldots,x_{n}, \bar x_{n})).
\end{equation}
We construct a weight-preserving randomized bijection from 
\begin{equation}
\bigcup_{\sigma \in \mathcal{I}_n} \mathcal{M}(\sigma  \, \T_{2n,k}^{\mathbf{p}}(x_1,\bar x_1,\ldots,x_{n}, \bar x_{n}))
\end{equation}
to 
\begin{equation}
\bigcup_{\sigma \in \mathcal{I}_n} \mathcal{M}(\sigma \,  \ST_{2n,k,j}^{\mathbf{p}}(x_1,\bar x_1,\ldots,x_{n}, \bar x_{n})).
\end{equation}
Note that each element in the two unions consists essentially of a pair of a matching of $\T_{2n,k}^{\mathbf{p}}$ and a certain weighting of the edges of the graph which is prescribed by $\sigma$.

We explain the weight-preserving randomized bijection also with the help of the example given in Figure~\ref{T65_matching}. The first modification is already indicated in blue there: with the exception of the top row and the bottom row, each row of hexagons has \emph{precisely} two vertices of degree $2$ (top and bottom row have more vertices of degree two); corresponding vertices of degree $2$ also exist in the top and bottom row. We add paths of length $2$ at these vertices in rows $1,3,\ldots,2n-1$ and extend the matching accordingly. 

\begin{figure}
\scalebox{0.5}{
\psfrag{1}{\large$1$}
\psfrag{1'}{\large$\overline{1}$}
\psfrag{2}{\large$2$}
\psfrag{2'}{\large$\overline{2}$}
\psfrag{3}{\large$3$}
\psfrag{3'}{\large$\overline{3}$}
\includegraphics{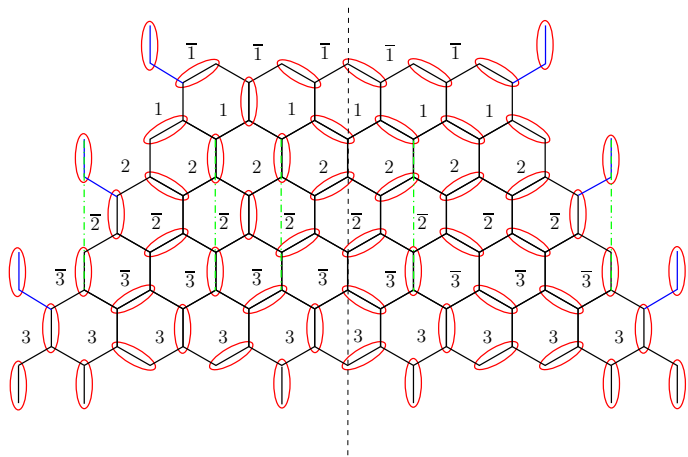}}
\caption{\label{T65_matching} The running example in the proof of Lemma~\ref{line}: a matching of $\T_{6,5}^{1,2,5,7,10,11}(\bar x_1,x_1,x_2,\bar x_2,\bar x_3,x_3)$ }\end{figure}

We consider the subgraph $R_{2i-1}$ that consists of a fixed odd row $2i-1$ of hexagons together with the edges just added so that $R_{2 i-1}$ has a half hexagon on the left side and a half hexagon on the right side. Further we consider the following subdivision of $R_{2 i-1}$ into sections that start and end with half hexagons: The dividing lines are the vertical lines that go through two vertices of $R_{2 i-1}$, where at least one of them is matched to a vertex not contained in $R_{2 i-1}$. In Figure~\ref{T65_matching}, these dividing lines are indicated in green for row $3$. In principle, there are three types of sections as indicated in Figure~\ref{threetypes}.

\begin{figure}
\scalebox{0.5}{\includegraphics{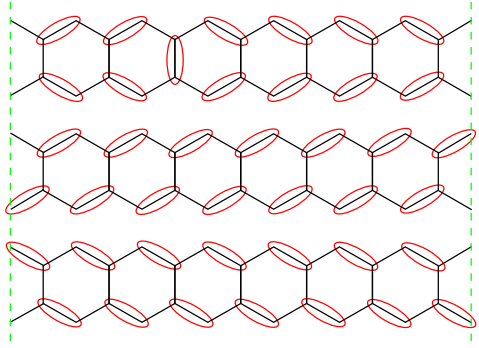}}
\caption{\label{threetypes} The three types of sections of rows: the first type has precisely one vertical matching edge, in the second type the matching edges are precisely those of type \protect\swne, while in the third type the matching edges are precisely those of type \protect\senw.}
\end{figure}

We now perform certain operations to the rows $1,3,\ldots,2n-1$. More specifically, we choose for each of these rows either the part left of the line $\ell_j$ or the part right of the line $\ell_j$ to which we apply the operations  then. The side is fixed if the two vertices of the row that lie on the line $\ell_j$ are matched to vertices that lie in the row and on the  same side of $\ell_j$; in that case we choose the other side. Otherwise the two sides are chosen with equal probability $\frac{1}{2}$. The line $\ell_j$ serves then also as a dividing line of the sections of $R_{2i -1}$ (possibly in addition to the lines that were already identified). 

For each odd row $2i-1$ and each section of type $1$ that is situated on the chosen side of that particular row, we now perform the following operation; see also Figure~\ref{type1}. If the section contains $t$ vertical edges and the unique vertical matching edge is in position $s$, counted from the left, then we transform this section into another section of type $1$, where the unique vertical matching edge is in position $t+1-s$. This surely changes the weight, however, this can be compensated by changing the weights in this section, more specifically we simply interchange weights $x_i$ with $\bar x_i$. We can change the weights accordingly in sections of type $2$ and $3$ on the chosen side, as the weight is $1$ for both choices in these sections.

\begin{figure}
\scalebox{0.5}{
\psfrag{i}{\large$i$}
\psfrag{j}{\large$\overline{i}$}
\includegraphics{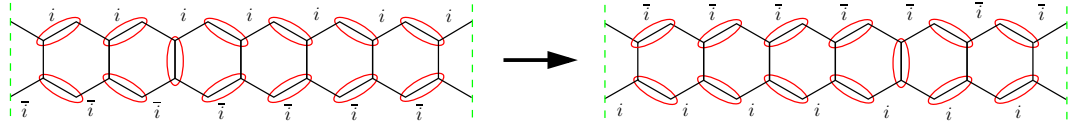}}
\caption{\label{type1} Reflecting the matching edges vertically and the edge weights horizontally leaves the total weight of $\bar x_i^{2}$ unchanged.}\end{figure}

Finally, we modify the weights on the right side of the line $\ell_j$ in such a way that the matching generating function is unchanged; see Figure~\ref{right} for an illustration. The modification is based on the following simple observation: Suppose $G$ is an edge-weighted graph and $v$ is a vertex of $V$. Let $G'$ be the edge-weighted graph obtained from $G$ by multiplying the weights of all edges incident with $v$ with $w$. Then we have the following relation between the two matching generating functions:
\begin{equation}
\m(G') = w \m(G).
\end{equation}
Now consider all vertical edges $e$ that are (fully) contained in row $R_{2i-1}$ and located right of the line $\ell_j$. Let $a$ and $b$ be the two endpoints of $e$. Both vertices have precisely two incident edges with weight $1$ (one of them is $e$), while the third edge has weight $x_i$ for one vertex and weight $\bar x_i$ for the other vertex (these edges are both of type \swne). We assume without loss of generality that there is an edge incident with $a$ that has weight $x_i$. Now we multiply the weights of each edge incident with $a$ by $\bar x_i$ and the weights of each edge incident with $b$ by $x_i$. Now $a$ has precisely one vertex incident with it that has weight $\bar x_i$, while $b$ has precisely one vertex incident with it that has weight $x_i$. These edges are both of type \senw, and we have reached an element of 
\begin{equation}
\bigcup_{\sigma \in \mathcal{I}_n} \mathcal{M}(\sigma  \ST_{2n,k,j}^{\mathbf{p}}(x_1,\bar x_1,\ldots,x_{n}, \bar x_{n})).
\end{equation}

\begin{figure}
\scalebox{0.5}{
\psfrag{i}{\large$i$}
\psfrag{j}{\large$\overline{i}$}
\includegraphics{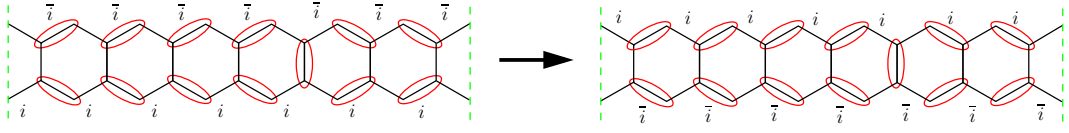}}
\caption{\label{right} Modification of edge weights on the right leaves the weights unchanged.}\end{figure}

Each step of the procedure is obviously invertible and thus this establishes a randomized bijection. The probabilities for transforming one element of the domain of the bijection into an element of the codomain is $\frac{1}{2^m}$ or $0$, where $m$ is the number of odd rows such that one of its vertices on the line $\ell_j$ is matched to the left while the other is matched to the right. It is crucial that the procedure of the matching does not change this number $m$.  
\end{proof}

\begin{rem} The randomized bijection can be transformed into a classical bijection from 
\begin{equation}
\{0,1\}^n \times \bigcup_{\sigma \in \mathcal{I}_n} \mathcal{M}(\sigma  \, \T_{2n,k}^{\mathbf{p}}(x_1,\bar x_1,\ldots,x_{n}, \bar x_{n})) 
\end{equation}
to 
\begin{equation}
\{0,1\}^n \times \bigcup_{\sigma \in \mathcal{I}_n} \mathcal{M}(\sigma \,  \ST_{2n,k,j}^{\mathbf{p}}(x_1,\bar x_1,\ldots,x_{n}, \bar x_{n})), 
\end{equation}
where the $i$-th letter of the $\{0,1\}$-sequence in $\{0,1\}^n$ encodes which side of $\ell_j$ we choose in row $2i-1$, $1 \le i \le n$, whenever there is actually a choice. 
When there is no choice in row $2i-1$, then the $i$-th letter of the word has no effect. The $\{0,1\}$-sequence itself is unchanged in the procedure. 
\end{rem}

\section{Application of Ciucu's factorization theorem and a combinatorial proof of  Theorem~\ref{thm:fact1}(\ref{item:part1})}
\label{sec_ciucu}

\subsection{Ciucu's theorem} For convenience, we recall Ciucu's factorization theorem for graphs with reflective symmetry \cite{Ciu97}. We assume that the edge-weighted graph $G$ has the following properties: 
\begin{itemize} 
\item It is planar, bipartite and connected.
\item It exhibits symmetry with respect to a vertical symmetry axis $\ell$ (including that the edge-weights are symmetric).
\item Removing the vertices on $\ell$ disconnects the graph.
\end{itemize} 
Without loss of generality, we assume that there are an even number of vertices on the symmetry axis. If not, $G$ has an odd number of vertices by symmetry and thus no perfect matching. We denote by $\n(G)$ half of the number of vertices on the symmetry axis and by $a_1,b_1,\ldots,a_{\n(G)},b_{\n(G)}$ the vertices on $\ell$ as they appear from top to bottom. We refer to the vertices in one vertex class of $G$ as the \emph{positive vertices}, while we refer to the vertices in the other vertex class as the \emph{negative vertices}.
 
For a vertex $v$ on $\ell$, we define two cutting operations: ``Cutting left of $v$'' means that we delete all 
incident edges left of $\ell$, while ``cutting right of $v$'' means that we delete all incident edges right of $\ell$.
 Now we define two subgraphs of $G$ as follows: We perform the cutting operation right of all positive $a_i$'s and negative $b_i$'s, and left of all negative $a_i$'s 
and positive $b_i$'s. Reduce the weights of the edges on $\ell$ by half and leave all other weights unchanged. We obtain two disconnected graphs, and denote by $G^{+}$ the left graph and by $G^{-}$ the right graph.
We are now able to state Ciucu's factorization theorem.
\begin{thm}\cite[Theorem~1.2]{Ciu97} With the notations introduced above, we have 
\begin{equation}
\m(G)  = 2^{\n(G)} \m(G^+) \m(G^-).
\end{equation} 
\end{thm}

\subsection{Proof of Theorem~\ref{thm:fact1}(\ref{item:part1})} We now apply the gathered information to our problem. Recall that we consider the following specialization of the Schur function 
$
s_{\widehat{\lambda}}(x_1,\bar x_1, \ldots, x_n,\bar x_n)
$
for 
\begin{equation}
\widehat{\lambda} = (\lambda_1+1,\lambda_2+1, \dots, \lambda_n + 1, 
 - \lambda_n,-\lambda_{n-1},\dots, - \lambda_1)+ \lambda_1.
\end{equation}
Using Theorem~\ref{schur-matching}, this is equal to the following 
matching generating function.
\begin{equation}
\label{nonsym}
\m (\T_{2n, 2 \lambda_1+1}^{\mathbf{p}}(x_1,\bar x_1,\ldots,x_n,\bar x_n))
\end{equation}
with 
\begin{equation}
\mathbf{p} = (1-\lambda_1,2-\lambda_2,\ldots,n-\lambda_n,n+2+\lambda_n,n+3+\lambda_{n-1},\ldots,2n+1+\lambda_1)+\lambda_1.
\end{equation}
Now we observe that with this particular choice of parameters, 
$\T_{2n, 2 \lambda_1+1}^{\mathbf{p}}$ is symmetric with respect to the axis $l_{\lambda_1+1}$. However, the weights are not symmetric. To remedy  this issue, we apply Lemma~\ref{line} to see that 
\eqref{nonsym} is equal to 
\begin{equation}
\label{sum}
\frac{1}{2^n} \sum_{\sigma \in {\mathcal I}_n} \m( \sigma 
\ST_{2n,2 \lambda_1+1,\lambda_1+1}^{\mathbf{p}} (x_1,\bar x_1,\ldots,x_n,\bar x_n)).
\end{equation}

Now the weighted graphs $\sigma 
\ST_{2n,2 \lambda_1+1,\lambda_1+1}^{\mathbf{p}} (x_1,\bar x_1,\ldots,x_n,\bar x_n)$ have symmetric edge weights and we may apply Ciucu's factorization theorem. For $n=2$ and $\lambda=(2,0)$, we have $\mathbf{p}=(1,4,6,9)$ and the graph $\ST_{4,5,3}^{1,4,6,9}(x_1,\x_1,x_2,\x_2)$ is displayed in Figure~\ref{ST45}.

We fix the top vertex on the symmetry axis to be a positive vertex (which is $a_1$ in the setting of the factorization theorem). The positive vertices alternate with the negative vertices on the symmetry axis.
By Ciucu's construction, we cut right of all the vertices on the symmetry axis.
It follows that $G^+$ is 
$\htp{\lambda_n+1,\lambda_{n-1}+2,\ldots,\lambda_1+n}{2n,\lambda_1}$, 
while $G^{-}$ is $\htm{\lambda_n+1,\lambda_{n-1}+2,\ldots,\lambda_1+n}{2n,\lambda_1}$. Thus we obtain 
\begin{equation}
2^n \m( \sigma \, \htp{\lambda_n+1,\lambda_{n-1}+2,\ldots,\lambda_1+n}{2n,\lambda_1}(x_1,\bar x_1,\ldots,x_n,\bar x_n)) \m( \sigma \, \htm{\lambda_n+1,\lambda_{n-1}+2,\ldots,\lambda_1+n}{2n,\lambda_1}(x_1,\bar x_1,\ldots,x_n,\bar x_n))
\end{equation}
for the summand of $\sigma \in  {\mathcal I}_n$ in \eqref{sum}. 

\begin{figure}
\scalebox{0.35}{
\psfrag{1}{\Large$1$}
\psfrag{2}{\Large$\bar 1$}
\psfrag{3}{\Large$2$}
\psfrag{4}{\Large$\bar 2$}
\includegraphics{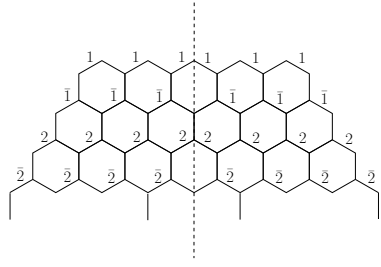}} 
\caption{\label{ST45} The graph $\ST_{4,5,3}^{1,4,6,9}(x_1,\bar x_1,x_2,\bar x_2)$.
}
\end{figure}

As for $\m(\htm{\lambda_n+1,\lambda_{n-1}+2,\ldots,\lambda_1+n}{2n,\lambda_1}(x_1,\bar x_1,\ldots,x_n))$, we know from Theorem~\ref{prop:symplectic-pattern-character} that it is equal to $sp_{\lambda}(x_1,\ldots,x_n)$. Now it is crucial that $sp_{\lambda}(x_1,\ldots,x_n)$ is invariant under replacing $x_i$ by $\x_i$ for any $i$. 
This follows easily, for instance, from the determinantal expression of 
$sp_{\lambda}(x_1,\ldots,x_n)$, but can also be seen from the combinatorial interpretations provided in this paper: Using Theorem~\ref{proc-symb}, one has to employ an operation analogous to $J_i$ as defined in Proposition~\ref{prop:OP-involution} for odd rows, or, alternatively, relying on Theorem~\ref{prop:symplectic-pattern-character}, one has to employ an operation analogous to one from 
Lemma~\ref{line}. Now, since $\m(\sigma \htm{\lambda_n+1,\lambda_{n-1}+2,\ldots,\lambda_1+n}{2n,\lambda_1}(x_1,\bar x_1,\ldots,x_n))$ is obtained from $\m( \htm{\lambda_n+1,\lambda_{n-1}+2,\ldots,\lambda_1+n}{2n,\lambda_1}(x_1,\bar x_1,\ldots,x_n, \bar x_n))$ by replacing $x_i$ by $\x_i$ for all $i$ such that $(x_i,\x_i)$ ``appears'' in $\sigma$, it follows that  
\begin{equation}
sp_{\lambda}(x_1,\ldots,x_n) = \m(\sigma \htm{\lambda_n+1,\lambda_{n-1}+2,\ldots,\lambda_1+n}{2n,\lambda_1}(x_1,\bar x_1,\ldots,x_n, \bar x_n))
\end{equation}
for all $\sigma$. We can conclude that \eqref{sum} is equal to 
\begin{equation}
sp_{\lambda}(x_1,\ldots,x_n) 
\sum_{\sigma \in {\mathcal I}_n} \m( \sigma \, \htp{\lambda_n+1,\lambda_{n-1}+2,\ldots,\lambda_1+n}{2n,\lambda_1}(x_1,\bar x_1,\ldots,x_n,\bar x_n)).
\end{equation}
The first part of Theorem~\ref{thm:fact1} now follows from Theorem~\ref{prop:even-ortho-pattern-character}.

\section{``Doubling'' the graph and a combinatorial proof of Theorem~\ref{thm:fact1}(\ref{item:part2})}
\label{sec_doubling}

As for the second part of Theorem~\ref{thm:fact1}, the largest part of the partition $\widehat{\lambda}$ is even, and therefore we need to work with a graph 
$\T_{n,k}^{\mathbf{p}}$ where there are an even number $k$ of hexagons in the top row (since $k$ is just the largest part of $\widehat{\lambda}$ by Theorem~\ref{schur-matching}). Therefore Lemma~\ref{line} cannot be applied directly to achieve symmetric edge weights, since the symmetry axis of the unweighted graph is not a line $\ell_j$ as described in the statement of Lemma~\ref{line}. Instead it will turn out to be useful to apply the operations provided in the following lemma, simultaneously at various places, with the effect that the graph is (almost) doubled.

\begin{lem} 
\label{replace}
The following replacement rule in a weighted graph leaves the matching generating function invariant, where in the replacement the degree of the black vertices does not change and the red vertices are the connecting points.  In the drawings, the label next to the edge indicates the weight and we assume $a_1 \cdot a_2 =a$, $b_1 \cdot b_2 =b$, $y_1 \cdot y_2 = y$ and 
$z_1 \cdot z_2=z$.
\begin{center}
\scalebox{0.3}{
\psfrag{a}{\Huge$a$}
\psfrag{b}{\Huge$b$}
\psfrag{a1}{\Huge$a_1$}
\psfrag{b1}{\Huge$b_1$}
\psfrag{a2}{\Huge$a_2$}
\psfrag{b2}{\Huge$b_2$}
\psfrag{x}{\Huge$y$}
\psfrag{y}{\Huge$z$}
\psfrag{t}{\Huge$t$}
\psfrag{x1}{\Huge$y_1$}
\psfrag{x2}{\Huge$y_2$}
\psfrag{y1}{\Huge$z_1$}
\psfrag{y2}{\Huge$z_2$}
\psfrag{ab}{\Huge$\frac{t}{a_1 z_1 + b_2 y_2}$}
\includegraphics{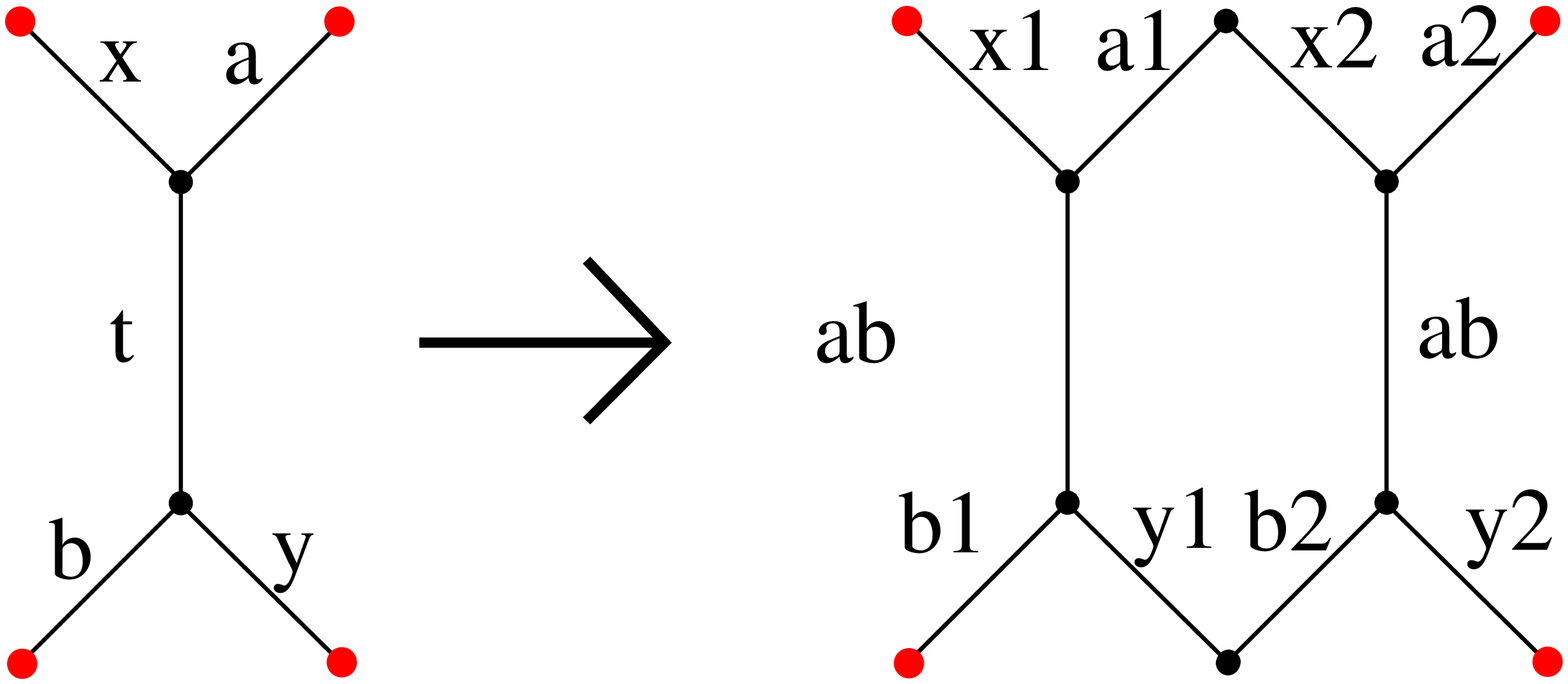}}
\end{center}
\end{lem}
\begin{proof} 
The proof proceeds by considering the various (not necessarily perfect) matchings of the subgraphs that cover all black vertices.
\end{proof}

We will apply the rules in Lemma~\ref{replace} to the graph $\T_{n,k}^{\mathbf{p}}(x_1,\ldots,x_n)$ at various places. If we say that we apply the rule to a particular vertical edge then we mean that the vertical edge on the left side of the rule corresponds to this chosen vertical edge. 
By $\DT_{n,k}^{\mathbf{p}}(x_1,\ldots,x_n)$ we denote the graph that is obtained from $\T_{n,k}^{\mathbf{p}}(x_1,\ldots,x_n)$ by applying the rules from Lemma~\ref{replace} to the vertical edges in all odd rows $1,3,5,\ldots$. In row $2i-1$, we specify the weights as follows: $a_1= x_{2i-1}^{1/2}$, $b_1= b_2 = x_{2i}^{1/2}$, $y_2=1$, $z_1=z_2=1$, $t=1$, and 
$a_2 \in \{0,x_{2i-1}^{1/2} \}$ and $y_1 \in \{0,1\}$, where we choose $y_1=0$ if we are on the left boundary and $a_2=0$ if we are on the right boundary.

The graph $\DT_{4,4}^{1,3,6,8}$ is displayed in Figure~\ref{double_graph}. Lemma~\ref{replace} implies
\begin{equation}
\label{dt}
\m(\T_{n,k}^{\mathbf{p}}(x_1,\ldots,x_n)) = 
\m(\DT_{n,k}^{\mathbf{p}}(x_1,\ldots,x_n)).
\end{equation}

\begin{figure}
\scalebox{0.35}{
\psfrag{1}{\Large$\widehat{1}$}
\psfrag{2}{\Large$\widehat{2}$}
\psfrag{3}{\Large$\widehat{3}$}
\psfrag{4}{\Large$\widehat{4}$}
\includegraphics{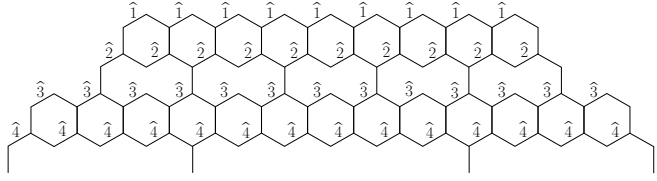}}
\caption{\label{double_graph} The graph $\DT_{4,4}^{1,3,6,8}(x_1,x_2,x_3,x_4)$. The vertical edges in row $1$ carry the weight $\left(x_1^{1/2}+x_2^{1/2}\right)^{-1}$, while the vertical edges in row $3$ carry the weight $\left(x_3^{1/2}+x_4^{1/2}\right)^{-1}$, and $\widehat{i}$ stands for $x_i^{1/2}$.}
\end{figure}

Again we consider also different edge weights of $\DT_{n,k}^{\mathbf{p}}(x_1,\ldots,x_n)$: Let $\ell_j$ denote the vertical line that contains the $j$-th vertex in the top row (counted from the left). Left of $\ell_j$, the weights are unchanged, but right of $\ell_j$ the edges of type \senw are assigned the weight $x_i^{1/2}$ in row $i$, while edges of type \swne are assigned the weight $1$; the weights of the vertical edges do not change. This edge-weighted graph is denoted by $\SDT_{n,k,j}^{\mathbf{p}}(x_1,\ldots,x_n)$. Now, a slight modification\footnote{Only the graphs differ insofar as some edges that are irrelevant for the procedure have been deleted, since they only appear in even rows. Also the edge weights of the vertical edges do not cause any difficulties as they are constant on rows.}
and 
of Lemma~\ref{line} shows that 
\begin{equation}
\label{sdt}
\m\left(\DT_{2n,k}^{\mathbf{p}}(x_1,\bar x_1,\ldots,x_n,\bar x_n) \right) =
\frac{1}{2^n} \sum_{\sigma \in {\mathcal I}_n} \m\left(\sigma \SDT_{2n,k,j}^{\mathbf{p}}(x_1, \bar x_1,\ldots, x_n, \bar x_n) \right).
\end{equation}
In the second part of Theorem~\ref{thm:fact1}, we consider the specialization of the Schur polynomial 
\begin{equation}
s_{\widehat{\lambda}}(x_1,\bar x_1,\ldots,x_n,\bar x_n),
\end{equation}
where 
\begin{equation}
\widehat{\lambda} = 
(\lambda_1,\ldots,\lambda_n,-\lambda_n,\ldots,-\lambda_1) + \lambda_1.
\end{equation}
By Theorem~\ref{schur-matching}, this is equal to the matching generating function 
\begin{equation}
\m(\T_{2n,2 \lambda_1}^{\mathbf{p}}(x_1, \bar x_1,\ldots,x_n, \bar x_n),
\end{equation}
with 
\begin{equation}
\mathbf{p} = (1- \lambda_1,2-\lambda_2,\ldots,n-\lambda_n,
n+1+\lambda_n,\ldots,2n+\lambda_1) + \lambda_1.
\end{equation}
Using \eqref{dt} and \eqref{sdt}, this is equal to 
\begin{equation}
\label{sum2}
\frac{1}{2^n} \sum_{\sigma \in {\mathcal I}_n} \m\left(\sigma \SDT_{2n,2 \lambda_1 ,2 \lambda_1+1}^{\mathbf{p}}(x_1, \bar x_1,\ldots, x_n, \bar x_n) \right).
\end{equation}
Note that $\SDT_{2n,2 \lambda_1, 2 \lambda_1 +1}^{\mathbf{p}}(x_1, \bar x_1,\ldots,x_n,\bar x_n)$ has $4 \lambda_1 +1$ hexagons in the top row and thus $\ell_{2 \lambda_1 +1}$ is its symmetry axis.

Now we apply Ciucu's factorization theorem to each summand in \eqref{sum2}. We declare the top vertex $a_1$ on the symmetry axis to be a positive vertex (in the context of Ciucu's theorem). It follows that all vertices $a_i$ on the symmetry axis are positive, while all vertices $b_i$ on the symmetry axis are negative. Thus we need to cut right of all vertices on the symmetry axis.

We first consider the left graphs $G^+$. The vertices on the former symmetry axis are all of degree $1$. We delete the incident edges as well as all edges adjacent to these edges to obtain a new graph whose matching generating function differs from the original by a multiplicative factor of 
\begin{equation}
\sigma \prod_{i=1}^{n} x_i^{1/2},
\end{equation}
which needs to be multiplied to the generating function of the reduced graph to obtain the generating function of the former graph. 
For our example, the graph is displayed in Figure~\ref{gplus}(a) for $\sigma = \id$. Next we apply Lemma~\ref{replace} in the reverse direction, i.e., we shrink hexagons to vertical edges. To be more precise, we shrink the hexagons in the odd rows that are in odd positions, if counted from the left. We obviously obtain the graph 
\begin{equation}
\sigma \, \hhtp{\lambda_n,\lambda_{n-1}+1,\dots,\lambda_1+n-1}
 {2n,\lambda_1-1}(x_1,\bar x_1,\ldots,x_n,\bar x_n)
\end{equation}
in Figure~\ref{gplus}(b) and so the matching generating function of $G^+$ is 
\begin{equation}
\left[ \sigma \prod_{i=1}^{n} x_i^{1/2} \right]
 M \left( \sigma \, \hhtp{\lambda_n,\lambda_{n-1}+1,\dots,\lambda_1+n-1}
 {2n,\lambda_1-1} (x_1, \bar x_1,\ldots, x_n, \bar x_n) \right).
 \end{equation}

\begin{figure}
\begin{tabular}{cc}
\scalebox{0.35}{
\psfrag{1}{\Large$\widehat{1}$}
\psfrag{2}{\Large$\overline{\widehat{1}}$}
\psfrag{3}{\Large$\widehat{2}$}
\psfrag{4}{\Large$\overline{\widehat{2}}$}
\includegraphics{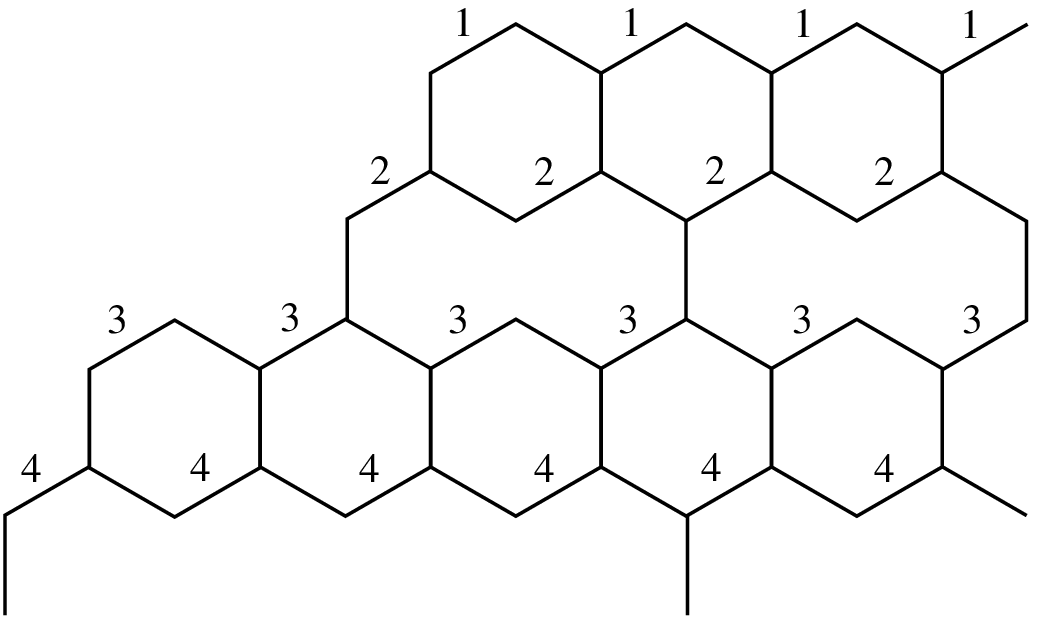}} \qquad \qquad & \qquad \qquad  
\scalebox{0.35}{
\psfrag{1}{\Large$1$}
\psfrag{2}{\Large$\bar 1$}
\psfrag{3}{\Large$2$}
\psfrag{4}{\Large$\bar 2$}
\includegraphics{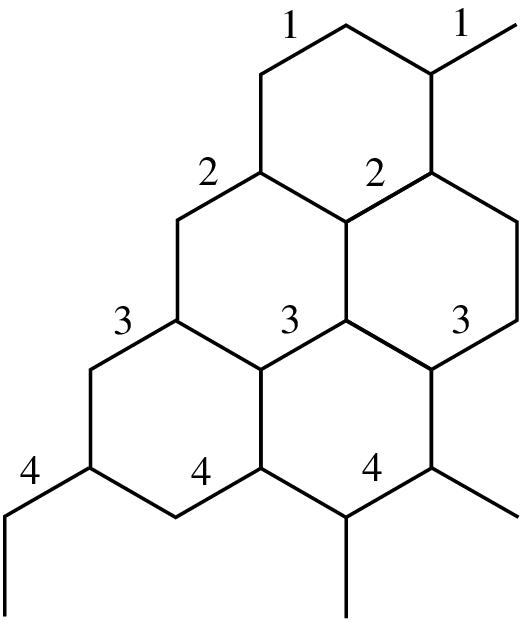}} \\
(a) & (b)
\end{tabular}
\caption{\label{gplus} (a) The graph $G^+$ of  $\SDT_{4,4,5}^{1,3,6,8}(x_1,\bar x_1,x_2, \bar x_2)$. The vertical edge-weights in odd rows are as in Figure~\ref{double_graph}. (b) The graph $G^+$ after application of Lemma~\ref{replace} in reverse direction.}
\end{figure}

We now consider the right graphs $G^-$. For our example, the graph is displayed in Figure~\ref{gminus} when $\sigma=\id$. Again we apply Lemma~\ref{replace} in reverse direction to every other hexagon in odd rows, more precisely to the hexagons in odd positions if counted from the \emph{right}. The weights of the vertical edges are then $1$ again, except for the leftmost vertical edges of the odd rows which still carry the weight $x_i^{1/2} + x_i^{-1/2}$ in row $2i-1$. 
We multiply the weights of the edges incident with the vertices along the left horizontal line in positions $2i$, $1 \le i \le n$, counted from the top with  
$x_i^{1/2} + x_i^{-1/2}$. We obtain the graph 
\begin{equation}
\sigma \, \hhtm{\lambda_n+1,\lambda_{n-1}+2,\dots,\lambda_1+n}
 {2n,\lambda_1}(x_1,\bar x_1,\ldots,x_n,\bar x_n).
\end{equation}
Using Theorem~\ref{oddorth}, the matching generating function is equal to 
$\sigma \oo_\lambda(x_1,\ldots,x_n)$ and thus independent of $\sigma$ as 
$ \oo_\lambda(x_1,\ldots,x_n)$ is invariant under replacing $x_i$ with $\bar x_i$. In total, we obtain 
\begin{multline}
\prod_{i=1}^{n} \left( x_i^{1/2} + x_i^{-1/2} \right)^{-1} \oo_\lambda(x_1,\ldots,x_n) \\ \times \sum_{\sigma \in {\mathcal I}_n} 
\left[ \sigma \prod_{i=1}^{n} x_i^{1/2} \right] 
 M \left(\sigma \, \hhtp{\lambda_n,\lambda_{n-1}+1,\dots,\lambda_1+n-1}
 {2n,\lambda_1-1} (x_1, \bar x_1,\ldots, x_n, \bar x_n) \right).
\end{multline}
Theorem~\ref{thm:fact1}(\ref{item:part2}) now follows from Theorem~\ref{prop:even-ortho-pattern-character-half}.

\begin{figure}
\begin{tabular}{cc}
\scalebox{0.35}{
\psfrag{1}{\Large$\widehat{1}$}
\psfrag{2}{\Large$\overline{\widehat{1}}$}
\psfrag{3}{\Large$\widehat{2}$}
\psfrag{4}{\Large$\overline{\widehat{2}}$}
\includegraphics{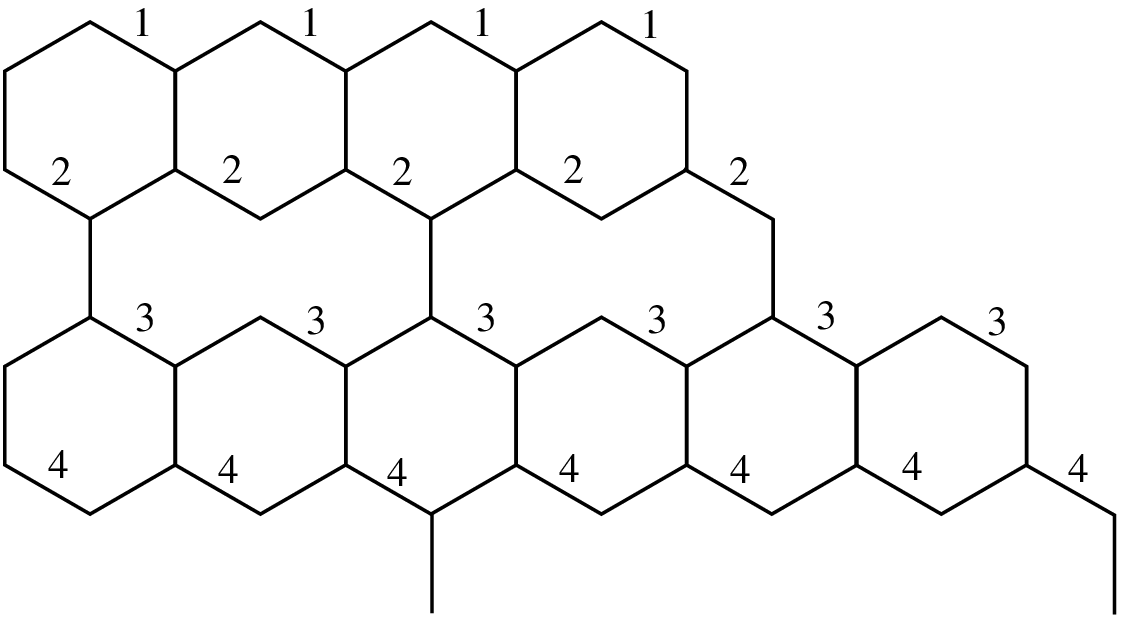}} \qquad \qquad & \qquad \qquad  
\scalebox{0.35}{
\psfrag{1}{\Large$1$}
\psfrag{2}{\Large$\bar 1$}
\psfrag{5}{\Large$I+\bar 1$}
\psfrag{3}{\Large$2$}
\psfrag{4}{\Large$\bar 2$}
\psfrag{6}{\Large$I+\bar 2$}
\includegraphics{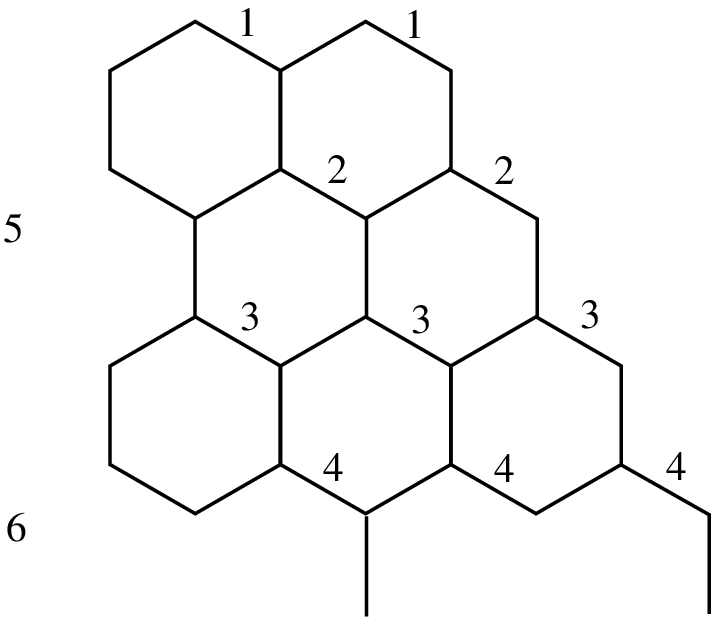}} \\
(a) & (b)
\end{tabular}
\caption{\label{gminus} (a) The graph  $G^-$ of  $\SDT_{4,4,5}^{1,3,6,8}(x_1,\bar x_1,x_2, \bar x_2)$. The vertical edge-weights in odd rows are as in Figure~\ref{double_graph}. (b) The graph $G^-$ after application of Lemma~\ref{replace} in reverse direction. Here we denote the number $1$ by $I$.}
\end{figure}

\section{The skew case: Proof of Theorem~\ref{skew_theo}}
\label{sec_skew}

Our bijective proof of Theorem~\ref{thm:fact1} actually reveals that the theorem can be generalized to skew Schur polynomials, provided that also the inner shape of the skew shape satisfies a certain symmetry property as laid down in Theorem~\ref{skew_theo}. This follows merely from the observation that edges can also stick out at the top of the trapezoidal graph $\T_{n,k}$, just as they stick out at the bottom according to the outer shape. These top edges encode the inner shape of the skew shape. The symmetry property of the inner shape has to guarantee that also the distribution of the top edges is symmetric with respect to the vertical symmetry axis of the graph so that Ciucu's factorization theorem can be applied. Phrased differently, this can also be seen as a refinement of Theorem~\ref{thm:fact1}, where we fix in a certain fixed row, say, $m$ of the graph $\T_{n,k}$ the vertical matching edges. (In the associated semistandard tableaux  this corresponds to fixing the shape of the entries that are less than or equal to $m$.) From this point of view, the skew case concerns the subgraph of $\T_{n,k}$ consisting of what is below this fixed row and adding edges sticking out at the (new) top according to the fixed matching in (the old) row $m$. 
Since our ``procedures'' (in a sense) do not mix between different rows of $\T_{n,k}$ (especially those used in the proof of Lemma~\ref{line}), the proofs of the straight cases generalize easily to the skew cases. The only additional effort is in finding the correct skew generalizations of the objects such as patterns and honeycomb graphs, which is also straightforward as we only need to chop off the appropriate number of top rows. This considerably increases the notational complexity because we also need to involve the positions of the extra edges sticking out at the top of the honeycomb graphs. A detailed proof of Theorem~\ref{skew_theo} will only obscure the main ideas. We therefore chose to give a detailed proof in the straight case and highlight a few major steps of the skew case in this section.

\subsection{Graphical model of skew Schur polynomials} 
Theorem~\ref{schur-matching} can be extended to \emph{skew Schur polynomials}. 
For this purpose, we only need to generalize $\T_{n,k}^{\mathbf{p}}(x_1,\ldots,x_n)$ so that we allow also additional vertical edges attached to the vertices in the topmost row. Let $n,k$ be a positive integer and $m$ be a non-negative integers with $m<n$ and $m \le k$, and consider 
$\T_{n-m,k}$. Adding $n$ vertical edges to a selection of the $n-m+k$ vertices in the bottommost row, while adding $m$ vertical edges to a selection of the $k$ vertices in the topmost row results in a bipartite graph that has the same number of vertices in each vertex class, and may as such possess a perfect matching. Let 
$\mathbf{p}=(p_1,\ldots,p_n)$, $1 \le p_1 < p_2 < \ldots < p_n \le n-m+k$, be these positions at the bottom, and, $\mathbf{q}=(q_1,\ldots,q_m)$, 
$1 \le q_1 < q_2 < \ldots < q_{m} \le k$, be these positions at the top, then 
$\mystrut^{\mathbf{q}}\T_{n-m,k}^{\mathbf{p}}$ denotes the corresponding (unweighted) graph, while 
$\mystrut^{\mathbf{q}}\T_{n-m,k}^{\mathbf{p}}(x_1,\ldots,x_{n-m})$ denotes the corresponding weighted graph. An example is given in Figure~\ref{skew}. Recall that for two partitions $\mu, \lambda$ such that the Young diagram of $\mu$ is contained in the Young diagram of $\lambda$, the skew Schur polynomial $s_{\lambda / \mu}(x_1,\ldots,x_{n-m})$ is the generating function of semistandard fillings of the Young diagram of shape $\lambda / \mu$ with respect to the weight in \eqref{tableau-weight} (replacing $n$ by $n-m$ there). The generalization of Theorem~\ref{schur-matching} is as follows.

\begin{figure}
\scalebox{0.35}{
\psfrag{1}{\Large$1$}
\psfrag{2}{\Large$2$}
\psfrag{3}{\Large$3$}
\psfrag{4}{\Large$4$}
\psfrag{5}{\Large$5$}
\psfrag{6}{\Large$6$}
\includegraphics{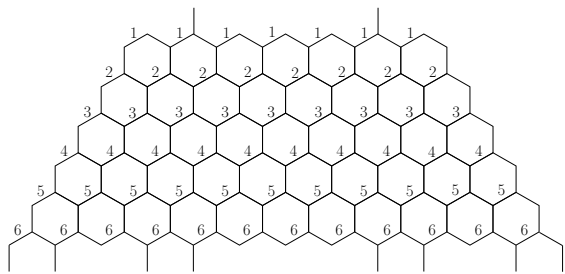}}
\caption{\label{skew} The graph $\mystrut^{2,6} \T_{6,7}^{1,2,4,5,9,10,12,13}(x_1,x_2,x_3,x_4,x_5,x_6)$}
\end{figure}

\begin{thm} 
\label{schur-matching-general}
Let $m,n$ be non-negative integers with $m<n$. Let $\lambda=(\lambda_1,\ldots,\lambda_n)$ and $\mu=(\mu_1,\ldots,\mu_{m})$ be partitions such that $\mu$ is contained in $\lambda$. Then 
\begin{equation}
s_{\lambda / \mu}(x_1,\ldots,x_{n-m}) = \m(^{\mu_{m}+1,\mu_{m-1}+2,\ldots,\mu_{1}+{m}}\T_{n-m,\lambda_1+m}^{\lambda_n+1,\lambda_{n-1}+2,\ldots,\lambda_{1}+n}(x_1,\ldots,x_{n-m})).
\end{equation}
\end{thm}

Note that there is the following ambiguity: Since we can add any number of zeros to $\mu$ and $\lambda$, there is of course an infinite family of graphs that can be used for a particular skew Schur polynomial. However, adding a zero to both $\mu$ and $\lambda$ means that we add a $\nearrow$-diagonal of hexagons left of the graph and we have an edge sticking out at the first position of the top row as well as at the first position of the bottom row. By forcing, the leftmost vertical edge of each row is a matching edge, thus the leftmost $\nearrow$-diagonal can be deleted again without changing the matching generating function.

\subsection{Combinatorial definitions of skew characters of other classical groups}

Here we first define skew versions of the characters of symplectic groups and of the orthogonal groups as they are suggested by our combinatorial proof, thereby clarifying the right-hand sides of the identities in Theorem~\ref{skew_theo}.  More specifically, we need to generalize the patterns provided in Definitions~\ref{def:nsympl}, \ref{defi_evenpatterns}, \ref{defi_splitpatterns} and their weights. The skew characters are then the generating functions of these patterns with respect to the weights. The general principle is very simple as laid down in the introductory paragraph of this section: In all cases, a pattern associated with the parameters $m,n$ ($m,n$ are as usual the number of parts of $\mu, \lambda$, respectively) is obtained by deleting at the top of a pattern of ``order'' $n$ a pattern of order $m$, except for the bottom row of the pattern of order $m$ which remains and corresponds to the inner shape. The generalized patterns are then obviously of trapezoidal shape.

We then also give references to appearances of these skew characters in representation theory, where they show up when restricting the straight characters to a certain subgroup and combinatorial interpretations 
are given in terms of various skew tableaux by Koike and Terada \cite{KoiTer90}. Using standard arguments, these tableaux representations can be transformed into pattern representations, thereby establishing the connection to the patterns appearing in our proof of Theorem~\ref{skew_theo}. 
To unify the relationship between our graphical models, patterns and tableaux, we will use notations that will differ from theirs and we will point to the difference in each case separately.
It is worth noting that Hamel~\cite{Ham97} has given determinantal formulas for skew symplectic and skew odd orthogonal characters.

\subsubsection{Skew symplectic characters} In this section, we assume $m < n$ to be non-negative integers, and $\mu, \lambda$ to be partitions with $m$ and $n$ parts, respectively, where we allow (as usual) also zero parts. 
The proof of Theorem~\ref{skew_theo} suggests the following definition, which is a generalization of Definition~\ref{def:nsympl}.

\begin{defi}[Trapezoidal symplectic patterns]
\label{trapez_sympl} 
Let $m,n,\mu,\lambda$ be as above.
\begin{enumerate}
\item A $(2n)/(2m)$-symplectic pattern has the shape of a $(2n)$-symplectic pattern with a 
$(2m)$-symplectic pattern deleted from the top, except for the bottom row of the $(2m)$-symplectic pattern that remains, such that the entries are non-negative integers and weakly increase along $\nearrow$-diagonals and $\searrow$-diagonals.
\item The rows are indexed from $0$ to $2n-2m$, starting with the top row, and $r_i$ is the sum of entries in row $i$. The weight of the pattern $P$ is then 
\begin{equation}
\ws P = \prod_{i=1}^{n-m} x_i^{r_{2i}-2r_{2i-1}+r_{2i-2}}.
\end{equation}
\item Denote the set of all $(2n)/(2m)$-symplectic patterns with top row $\mu$ and bottom row $\lambda$, both written in increasing order, as $\SP_{\lambda/\mu}$. 
\end{enumerate}
\end{defi} 

Note that our definitions are consistent with the straight case, which is obtained by setting $m=0$. In this case, the $0$-indexed row is empty, which is consistent with setting $r_0=0$ in Definition~\ref{def:nsympl}.
 The skew symplectic character $\sp_{\lambda/\mu}(x_1,\ldots,x_{n-m})$ appearing in Theorem~\ref{skew_theo} is the generating function 
\begin{equation}
\sum_{P \in \SP_{\lambda/\mu}} \ws P, 
\end{equation}
as can be seen when generalizing our combinatorial proof of Theorem~\ref{thm:fact1} to skew shapes. By Proposition~\ref{prop:skew-symp-patterns}, this combinatorially motivated definition coincides with the representation theoretic definition.

\begin{rem} 
\label{containment}
Since only now we have clarified the right-hand side of \eqref{toshow_skew}, let us point out the following subtlety: letting $m=2, n=3, \mu=(1,1), \lambda=(3,2,2)$, the shape $\widehat{\mu}$ is not contained in $\widehat{\lambda}$, so the left-hand side of  \eqref{toshow_skew} is zero. However, this is also true for the right-hand side, because 
$\SP_{\lambda/\mu}$ is empty in this case.
\end{rem}

In representation theory, the skew symplectic character $\sp_{\lambda/\mu}(x_1,\ldots,x_{n-m})$ appears when restricting the (straight) symplectic character to a certain subgroup.  In \cite[Proposition 4.1]{KoiTer90}, Koike and Terada provide a combinatorial interpretation of this skew symplectic character in terms of the generating function of the following skew tableaux.
The barred and unbarred symbols are interchanged in our convention.

\begin{defi}
Let $m,n,\mu,\lambda$ be as above and assume $\mu \subseteq \lambda$.
A {\em skew symplectic semistandard tableau $T$ of $\lambda / \mu$ with entries}
\begin{equation}
\overline{1} < 1 <  \overline{2} < 2 <  \cdots <  \overline{n-m} < n-m
\end{equation}
is a filling of the shape with these entries satisfying the following conditions:
\begin{itemize}
\item the entries increase weakly along rows,
\item the entries increase strictly along columns,
\item the entries in the $(i+m)$-th row must be greater than or equal to $\overline{i}$.
\end{itemize}
For $\alpha$ being any of the entries above, let $n_\alpha(T)$ be the number of occurrences of $\alpha$ in $T$. Then the weight of such a tableau is given by
\begin{equation}
\ws T = \prod_{i=1}^{n-m} x_i^{n_{i}(T) - n_{\overline{i}}(T)}.
\end{equation}
\end{defi}

Let $\SPT_{\lambda/\mu}$ be the set of skew symplectic semistandard tableau of 
shape $\lambda / \mu$ filled with entries $\overline{1},1,\ldots,\overline{n-m},n-m$.
For example, choosing $n=5, m=3$, $\mu=(3,2,1)$ and $\lambda=(5,4,4,2,2)$, then 
\begin{equation}
\label{eg:spt}
\begin{ytableau}
\none & \none & \none &\overline{1} & 1 \\
\none & \none & 1 & 1 \\
\none & 1 & 2 & 2 \\
\overline{1} & \overline{2} \\
\overline{2} & 2
\end{ytableau}
\end{equation}
is a skew symplectic tableau of shape $\lambda / \mu$ with weight $x_1^2 x_2$.

\begin{thm}[{\cite[Proposition 4.1]{KoiTer90}}]
\label{thm:koi-ter-skew-symp}
Let $m,n,\mu,\lambda$ be as above and assume $\mu \subseteq \lambda$.
The skew symplectic character of shape $\lambda / \mu$ is given by
\begin{equation}
\sp_{\lambda / \mu}(x_1,\dots,x_{n-m}) = \sum_{T \in \SPT_{\lambda/\mu}} \ws T.
\end{equation}
\end{thm}

The following result generalizes Theorem~\ref{proc-symb} and should be well-known although we could not find an explicit reference.
\begin{prop}
\label{prop:skew-symp-patterns}
Let $m,n,\mu,\lambda$ be as above. 
\begin{equation}
\sp_{\lambda/\mu}(x_1,\dots,x_{n-m}) = \sum_{P \in \SP_{\lambda/\mu}} \ws P.
\end{equation}
\end{prop}

\begin{proof}
Using Theorem~\ref{thm:koi-ter-skew-symp}, it suffices to find a a weight preserving bijection between 
skew symplectic tableaux and trapezoidal symplectic patterns. The strategy of proof is standard, 
using the general principle used to transform a semistandard tableau into a Gelfand-Tsetlin pattern as explained after Proposition~\ref{schur-matching} with appropriate modifications as follows. 
Recall that the usual semistandard tableau of shape $\lambda/\mu$ is in natural bijection with a sequence of
partitions
\begin{equation}
\mu = \lambda^0 \subseteq \lambda^1 \subseteq \cdots \subseteq \lambda^n = \lambda,
\end{equation}
such that $\lambda^i/\lambda^{i-1}$ is a horizontal strip for each $1 \leq i \leq n$. The entries in the shape 
$\lambda^i/\lambda^{i-1}$ are precisely those filled by $i$ in the tableau.
For skew sympectic semistandard tableau, the only difference is that the entries in the shape $\lambda^i/\lambda^{i-1}$
are filled by $\overline{(i+1)/2}$ if $i$ is odd and $i/2$ if $i$ is even. We now arrange each $\lambda^i$ in increasing order
to form the rows of the trapezoidal symplectic pattern. For the example of the skew symplectic pattern in
\eqref{eg:spt}, we obtain the pattern 
\begin{equation}
\begin{array}{cccccccccc}
& 1 && 2 && 3 &&&& \\
1 && 1 && 2 && 4 &&& \\
& 1 && 2 && 4 && 5 && \\
1 && 2 && 2 && 4 && 5 &  \\
& 2 && 2 && 4 && 4 && 5
\end{array}.
\end{equation}
The weights also match, completing the proof.
\end{proof}

\subsubsection{Skew even orthogonal characters}
In this section, we assume $m < n$ to be non-negative integers, and $\mu, \lambda$ to be partitions or half-integer partitions with $m$ and $n$ parts, respectively. 
The proof of Theorem~\ref{skew_theo} suggests the following definition, which is a generalization of Definition~\ref{defi_evenpatterns}.

\begin{defi}[Trapezoidal orthogonal patterns]
\label{trapez_evenortho}
Let $m,n,\mu,\lambda$ be as above.
\begin{enumerate}
\item 
A $(2n-1)/(2m-1)$-orthogonal pattern has the shape of a $(2n-1)$-orthogonal pattern with a 
$(2m-1)$-orthogonal pattern deleted from the top, except for the bottom row of the $(2m-1)$-orthogonal pattern that remains, such 
that the following conditions are satisfied:
\begin{itemize}
\item the entries are either all integers or all half-integers,
\item all entries except for the odd starters are non-negative,
\item the absolute values of the entries are weakly increasing along $\nearrow$-diagonals and $\searrow$-diagonals.
\end{itemize}
\item The rows are indexed from $-1$ to $2n-2m-1$, starting at the top, and $r_i^+$ is the sum of the absolute values of the entries in row $i$.
The weight of the pattern $P=(P_{i,j})$ is then 
\begin{equation}
\woe P = \prod_{i=1}^{n-m} x_i^{\sgn(P_{2i-1,1}) \sgn(P_{2i-3,1}) (r_{2i-1}^+ - 2r_{2i-2}^+ + r_{2i-3}^+)}.
\end{equation}
\item Denote the set of $(2n-1)/(2m-1)$-orthogonal patterns with top row $\mu$ or $\mu^-$ and bottom row $\lambda$ or $\lambda^-$, both written in increasing order, as $\OP_{\lambda/\mu}$. 
\end{enumerate}
\end{defi}

The even orthogonal character $\oe_{\lambda/\mu}(x_1,\ldots,x_{n-m})$ appearing in Theorem~\ref{skew_theo} is the the generating function 
\begin{equation}
\sum_{P \in \OP_{\lambda/\mu}} \woe P.
\end{equation}
By Proposition~\ref{prop:skew-even-orthog-patterns}, the combinatorially motivated definition coincides with the representation theoretic definition.

\begin{rem} 
Continuing the theme of Remark~\ref{containment}, observe that for $m=2, n=3, \mu=(1,1), \lambda=(3,2,2)$, the right-hand side of  \eqref{toshow_skew} is also zero because $\OP_{(\lambda+1)/(\mu+1)}$ is empty in this case.
\end{rem}

Assuming $\mu, \lambda$ to be integer partitions, Koike and Terada \cite[Proposition 4.3]{KoiTer90} show that the representation theoretical skew even orthogonal characters $\oe_{\lambda/\mu}(x_1,\ldots,x_{n-m})$, which appear when restricting the (straight) even orthogonal characters to a certain subgroup, have a combinatorial interpretation in terms of the generating function of the following skew tableaux.
Their symbols \musSharp{}$_i$, \musFlat{}$_i$, $i, \overline{i}$ correspond to our symbols 
$\hat{i}, \check{i}, \overline{i}, i$ respectively. 

\begin{defi}
\label{def:skew-oe-tableaux}
Let $m,n,\lambda, \mu$ be as above such that $\mu, \lambda$ are integer partitions and $\mu \subseteq \lambda$.
A \emph{skew even orthogonal semistandard tableau $T$ of shape $\lambda / \mu$} with entries
\begin{equation}
\hat{1} < \check{1} < \overline{1} < 1 < \hat{2} < \check{2} < \overline{2} < 2 < \cdots < \widehat{n-m} < \widecheck{n-m} < \overline{n-m} < n-m,
\end{equation}
is a filling of the shape with these entries satisfying the following conditions:
\begin{itemize}
\item the entries increase weakly along rows,
\item the entries increase strictly along columns,
\item the entries in the $(i+m)$-th row must be greater than or equal to $\check{i}$,
\item the entry $\hat{i}$ can only appear in the first column of row $i+m-1$,
 and $\check{i}$ can only appear in the first column of row $i+m$, and either of them only appears if the other does,
\item if $\overline{i}$ appears in the first column of row $i+m$ and also $i$ appear in that row, then there is an
$\overline{i}$ immediately above this $i$.
\end{itemize}
The {\em weight} of such a tableau is given by
\begin{equation}
\woe T = \prod_{i=1}^{n-m} x_i^{n_i(T) - n_{\overline{i}}(T)}.
\end{equation}
The set of skew even orthogonal semistandard tableau of shape $\lambda / \mu$ filled with entries $\hat{1},\dots,n-m$ is denoted by $\EOT_{\lambda/\mu}$.
\end{defi}

\begin{example}
\label{eg:skew-oe-111/1}
Let $n=3$ and $m=1$ and consider the skew shape $(1,1,1)/(1)$.  Then we have the following tableaux in $\EOT_{(1,1,1)/(1)}$ with entries in $\{\hat{1}, \check{1}, \overline{1}, 1, \hat{2}, \check{2}, \overline{2}, 2\}$:
\begin{align}
\begin{ytableau}
\emptyset \\
\overline{1} \\
\overline{2}
\end{ytableau}
\quad
\begin{ytableau}
\emptyset \\
\overline{1} \\
2
\end{ytableau}
\quad
\begin{ytableau}
\emptyset \\
1 \\
\overline{2}
\end{ytableau}
\quad
\begin{ytableau}
\emptyset \\
1 \\
2
\end{ytableau}
\quad
\begin{ytableau}
\emptyset \\
\overline{2} \\
2
\end{ytableau}
\quad
\begin{ytableau}
\emptyset \\
\hat{2} \\
\check{2}
\end{ytableau}
\end{align}
Here $\emptyset$ means that the cell is unoccupied.
It follows that $\oe_{(1,1,1)/(1)}(x_1,x_2) = 2 + x_1 x_2 + x_1 \bar x_2 + \bar x_1 x_2 + \bar x_1 \bar x_2$.
\end{example}

\begin{thm}[{\cite[Proposition 4.3]{KoiTer90}}]
\label{thm:koi-ter-skew-oe}
Let $m,n,\lambda, \mu$ be as above such that $\mu$ and $\lambda$ are integer partitions and $\mu \subseteq \lambda$.
The skew even orthogonal character of the shape $\lambda / \mu$ is given by
\begin{equation}
\oe_{\lambda / \mu}(x_1,\dots,x_{n-m}) = \sum_{T \in \EOT_{\lambda/\mu}} \woe T.
\end{equation}
\end{thm}

This theorem implies the following. Although this should be well-known as well, we again sketch the proof because we were unable to find an explicit reference.

\begin{prop}
\label{prop:skew-even-orthog-patterns}
Let $m,n,\lambda, \mu$ be as above and $\mu \subseteq \lambda$. Then
\begin{equation}
\oe_{\lambda/\mu}(x_1,\dots,x_{n-m}) = \sum_{P \in \OP_{\lambda/\mu}} \woe P.
\end{equation}
\end{prop}

\begin{proof}
As in the proof of Proposition~\ref{prop:skew-symp-patterns}, we prove this by constructing a weight-preserving bijection
$\phi : \OP_{\lambda/\mu} \to \EOT_{\lambda/\mu}$.
Let $P \in \OP_{\lambda/\mu}$ and let $\lambda^i$ be the $i$-th row of $P$, written in decreasing order.
For all $1 \leq i \leq n-m$:
\begin{itemize}
\item If $\sgn(P_{2i-3,1}) \sgn(P_{2i-1,1}) > 0$, fill the cells of $\lambda^{2i-2}/\lambda^{2i-3}$ by $\overline{i}$
and those of $\lambda^{2i-1}/\lambda^{2i-2}$ by $i$.
\item If $\sgn(P_{2i-3,1}) \sgn(P_{2i-1,1}) < 0$, fill the cells of $\lambda^{2i-2}/\lambda^{2i-3}$ by $i$
and those of $\lambda^{2i-1}/\lambda^{2i-2}$ by $\overline i$. 
Suppose, at this stage, $\overline{i}$ occurs below $i$, in column $c$.
Then $i$ should be replaced by $\hat{i}$ and $\overline{i}$, by $\check{i}$, if $c =1$, and
these two entries should be interchanged if $c > 1$.
\end{itemize}

We now claim that $\phi(P)$ satisfies all the conditions in Definition~\ref{def:skew-oe-tableaux}. The first, second and fourth conditions are not difficult to prove. The entries $i$ and $\overline{i}$ are filled in $\phi(P)$ only while parsing rows $2i-2$ and $2i-1$
of $P$. The lengths of these rows are $m+i-1$ and $m+i$ respectively. Therefore, these entries cannot occur after row $m+i$.
This proves the third condition. Now, suppose $\phi(P)_{i+m,1} = \overline{i}$, then $\sgn(P_{2i-3,1}) \sgn(P_{2i-1,1}) < 0$ and cell $(i+m,1)$ is contained in $\lambda^{2i-1}/\lambda^{2i-2}$.
If, moreover, $\phi(P)_{i+m,j} = i$ for some $j > 1$, and $\phi(P)_{i+m-1,j} \neq \overline{i}$, that means cell $(i+m,j)$ 
is contained in $\lambda^{2i-2}/\lambda^{2i-3}$, which is clearly impossible. Hence, the fifth condition  holds. It is not difficult to construct the inverse map.
\end{proof}

We illustrate Proposition~\ref{prop:skew-even-orthog-patterns} with the following example.

\begin{example}
The bijection $\phi$ maps the $9/3$-orthogonal pattern on the left onto the skew even orthogonal tableau on the right: 
\begin{equation}
\begin{array}{ccccccccc}
1 && 2 \\
& 1 && 3 \\
-1 && 2 && 3 \\
& 1 && 2 && 4 \\
0 && 1 && 3 && 4 \\
& 1 && 2 && 4 && 4 \\
-1 && 1 && 2 && 4 && 4
\end{array}
\overset{\phi}{\longrightarrow}
\raisebox{1.0cm}{
\begin{ytableau}
\none & \none & 1 & 2 \\
\none & \overline{1} & \overline{2} & 3 \\
\overline{1} & 3 \\
\hat{3}  \\
\check{3}
\end{ytableau}
}
\end{equation}
\end{example}

\subsubsection{Skew odd orthogonal characters}
In this section, we assume $m < n$ to be non-negative integers, and $\mu, \lambda$ to be integer partitions or half-integer partitions with $m$ and $n$ parts, respectively. The proof of Theorem~\ref{skew_theo} suggests the following definition, which is a generalization of Definition~\ref{defi_splitpatterns}.

\begin{defi}[Trapezoidal split orthogonal patterns]
Let $m, n, \mu, \lambda$ be as above.
\begin{enumerate}
\item A $(2n)/(2m)$-split orthogonal pattern has the shape of a $(2n)$-split orthogonal pattern with a 
$2m$-split orthogonal pattern deleted from the top, except for the bottom row of the $2m$-split orthogonal pattern that remains, such that the following conditions are satisfied:
\begin{itemize}
\item the entries except for the odd starters are either all integers are all half-integers,
\item the entries are non-negative, 
\item the entries weakly increase along $\nearrow$-diagonals and $\searrow$-diagonals.
\end{itemize}
\item The rows are indexed from $0$ to $2n-2m$, starting with the top, and $r_i$ is the sum of entries in row $i$. 
The weight of a pattern is then 
\begin{equation}
\woo P = \prod_{i=1}^{n-m} x_i^{r_{2i}-2r_{2i-1}+r_{2i-2}}.
\end{equation}
\item Denote the set of $(2n)/(2m)$-split orthogonal patterns with top row $\mu$ and bottom row $\lambda$, both written in increasing order, as 
$\SOP_{\lambda/\mu}$.
\end{enumerate}
\end{defi}

For integer partitions $\lambda, \mu$, the odd orthogonal character $\oo_{\lambda/\mu}(x_1,\ldots,x_{n-m})$ appearing in Theorem~\ref{skew_theo} is the generating function 
\begin{equation}
\sum_{P \in \SOP_{\lambda/\mu}} \woo P.
\end{equation}
The symbols \musFlat{}$_i$, $i, \overline{i}$ in \cite{KoiTer90} correspond to our symbols 
$\hat{i}, \overline{i}, i$ respectively. 

\begin{defi}
\label{def:skew-oo-tableaux}
Let $m,n,\lambda, \mu$ be as above and 
$\mu \subseteq \lambda$. 
A {\em skew odd orthogonal semistandard tableau $T$ of shape $\lambda / \mu$ with entries}
\begin{equation}
\hat{1} < \overline{1} < 1 < \hat{2} < \overline{2} < 2 < \cdots < \widehat{n-m} < \overline{n-m} < n-m,
\end{equation}
is a filling of $\lambda / \mu$ with these entries satisfying the following conditions:
\begin{itemize}
\item the entries increase weakly along rows,
\item the entries increase strictly along columns,
\item the entries in the $(i+m)$-th row must be greater than or equal to $\widehat{i}$,
\item $\hat{i}$ can only appear in the first column of row $i+m$. 
\end{itemize}
The weight of such a tableau is given by
\begin{equation}
\woo T = \prod_{i=1}^{n-m} x_i^{n_i(T) - n_{\overline{i}}(T)}.
\end{equation}
Let $\OOT_{\lambda/\mu}$ be the set of skew odd orthogonal semistandard tableau of shape $\lambda / \mu$ filled with entries $\hat{1},\dots,n-m$.
\end{defi}

\begin{example}
\label{eg:skew-oo-111/1}
Let $n=3$ and $m=1$ and consider the skew shape $(1,1,1)/(1)$.  Then we have the following tableaux in $\OOT_{(1,1,1)/(1)}$ with entries in $\{\hat{1}, \overline{1}, 1, \hat{2}, \overline{2}, 2\}$:
\begin{align}
\begin{ytableau}
\emptyset \\
\widehat{1} \\
\widehat{2}
\end{ytableau}
\quad
\begin{ytableau}
\emptyset \\
\widehat{1} \\
\overline{2}
\end{ytableau}
\quad
\begin{ytableau}
\emptyset \\
\widehat{1} \\
2
\end{ytableau}
\quad
\begin{ytableau}
\emptyset \\
\overline{1} \\
\widehat{2}
\end{ytableau}
\quad
\begin{ytableau}
\emptyset \\
\overline{1} \\
\overline{2}
\end{ytableau}
\quad
\begin{ytableau}
\emptyset \\
\overline{1} \\
2
\end{ytableau}
\quad
\begin{ytableau}
\emptyset \\
1 \\
\widehat{2}
\end{ytableau}
\quad
\begin{ytableau}
\emptyset \\
1 \\
\overline{2}
\end{ytableau}
\quad
\begin{ytableau}
\emptyset \\
1 \\
2
\end{ytableau}
\quad
\begin{ytableau}
\emptyset \\
\overline{2} \\
2
\end{ytableau}
\end{align}
Here, $\emptyset$ means that the cell is unoccupied.
It follows that 
\begin{equation}
\oo_{(1,1,1)/(1)}(x_1,x_2) = \left(1 + \bar x_1 + x_1 \right)
\left(1 + \bar x_2 + x_2 \right) + 1.
\end{equation}
\end{example}

\begin{thm}[{\cite[Proposition 4.2]{KoiTer90}}]
\label{thm:koi-ter-skew-oo}
Let $m,n,\lambda, \mu$ be as above such $\mu, \lambda$ are integer partitions and 
$\mu \subseteq \lambda$.
The skew odd orthogonal character of the shape $\lambda / \mu$ is given by
\begin{equation}
\oo_{\lambda / \mu}(x_1,\dots,x_{n-m}) = \sum_{T \in \OOT_{\lambda/\mu}} \woo T.
\end{equation}
\end{thm}

\begin{prop}
\label{prop:skew-odd-orthog-patterns}
Let $m,n,\lambda, \mu$ be as above such $\mu, \lambda$ are integer partitions and 
$\mu \subseteq \lambda$. Then
\begin{equation}
\oo_{\lambda/\mu}(x_1,\dots,x_{n-m}) = \sum_{P \in \SOP_{\lambda/\mu}} \woo P.
\end{equation}
\end{prop}

\begin{proof}
We construct a weight-preserving bijection $\phi: \SOP_{\lambda/\mu} \to \OOT_{\lambda/\mu}$. The strategy is very similar to the proof of Proposition~\ref{prop:skew-symp-patterns}.

Let $P \in \SOP_{\lambda/\mu}$. The $i$-th row of $P$, when sorted in weakly decreasing order is a partition $\lambda_i$, where 
odd starters are rounded up.
The interlacing conditions ensure that $\lambda_{i-1} \subseteq \lambda_{i}$, with $\lambda_0 = \mu$ and $\lambda_{2n-2m} = \lambda$. We now construct a tableau $T$ of shape $\lambda/\mu$ as follows: 
\begin{itemize}
\item Fill the entries of $\lambda_{2i-1}/\lambda_{2i-2}$ by $\overline{i}$ for $1 \leq i \leq n-m$. If the starter in row $2i-1$ is a half-integer, replace the first entry in row $m+i$ by $\widehat{i}$.
\item Fill the entries of $\lambda_{2i}/\lambda_{2i-1}$ by $i$ for $1 \leq i \leq n-m$.
\end{itemize}

See Example~\ref{eg:skew-odd-orthog-pattern} for an example of this map.
It is easily seen that the map is weight-preserving and the inverse map is easily constructed as well, establishing the result.
\end{proof}

\begin{example}
\label{eg:skew-odd-orthog-pattern}
The bijection $\phi$ of Proposition~\ref{prop:skew-odd-orthog-patterns} maps the $10/4$-split orthogonal pattern on the left onto the skew odd orthogonal tableau on the right: 
\begin{equation}
\begin{array}{cccccccccc}
& 1 && 2 \\
0 && 2 && 2 \\
& 2 && 2 && 3 \\
\frac{3}{2} && 2 && 3 && 3 \\
& 2 && 2 && 3 && 4 \\
1 && 2 && 2 && 3 && 4 \\
& 2 && 2 && 2 && 3 && 4
\end{array}
\overset{\phi}{\longrightarrow}
\raisebox{1.0cm}{
\begin{ytableau}
\none & \none & 1 & 2 \\
\none & \overline{1} & \overline{2} \\
1 & 1 \\
\hat{2} & \overline{2}  \\
\overline{3} & 3
\end{ytableau}
}
\end{equation}
 with weight $x_1^2 \x_2^1 x_3^0$.
\end{example}

\section*{Acknowledgements}
We thank R. Behrend for useful discussions. We also thank the anonymous referees for very useful comments and references.
We acknowledge the hospitality of the Institut Mittag Leffler where part of this was done.
Arvind Ayyer was partially supported by UGC centre for Advanced Study grant and by Department of Science and Technology grant EMR/2016/006624. Ilse Fischer
 acknowledges support from the Austrian Science Foundation FWF, START grant Y463 and SFB grant F50.

\bibliography{asmpp}

\begin{thebibliography}{10}

\bibitem{AyyBeh18}
A.~Ayyer and R.~Behrend.
\newblock Factorization theorems for classical group characters, with
  applications to alternating sign matrices.
\newblock {\em J. Combin. Theory Ser. A}, 165:78--105, 2019.

\bibitem{AyyBehFis16}
A.~Ayyer, R.~Behrend, and I.~Fischer.
\newblock Extreme diagonally and antidiagonally symmetric alternating sign
  matrices of odd order, 2016.
\newblock preprint.

\bibitem{BehFisKon17}
R.~Behrend, I.~Fischer, and M.~Konvalinka.
\newblock Diagonally and antidiagonally symmetric alternating sign matrices of
  odd order.
\newblock {\em Adv. Math.}, 315:324--365, 2017.

\bibitem{BenKnu72}
E.~A. Bender and D.~E. Knuth.
\newblock Enumeration of plane partitions.
\newblock {\em J. Combinatorial Theory Ser. A}, 13:40--54, 1972.

\bibitem{Ciu97}
M.~Ciucu.
\newblock Enumeration of perfect matchings in graphs with reflective symmetry.
\newblock {\em J. Combin. Theory Ser. A}, 77(1):67--97, 1997.

\bibitem{CiuKra09}
M.~Ciucu and C.~Krattenthaler.
\newblock A factorization theorem for classical group characters, with
  applications to plane partitions and rhombus tilings.
\newblock In {\em Advances in combinatorial mathematics}, pages 39--59.
  Springer, Berlin, 2009.

\bibitem{cohn98}
H.~Cohn, M.~Larsen, and J.~Propp.
\newblock The shape of a typical boxed plane partition.
\newblock {\em New York J. Math.}, 4:137--165, 1998.

\bibitem{FulHar91}
W.~Fulton and J.~Harris.
\newblock {\em Representation theory}, volume 129 of {\em Graduate Texts in
  Mathematics}.
\newblock Springer-Verlag, New York, 1991.
\newblock A first course.

\bibitem{Ham97}
A.~M. Hamel.
\newblock Determinantal forms for symplectic and orthogonal {S}chur functions.
\newblock {\em Canad. J. Math.}, 49(2):263--282, 1997.

\bibitem{Kin76}
R.~C. King.
\newblock Weight multiplicities for the classical groups.
\newblock In {\em Group theoretical methods in physics ({F}ourth {I}nternat.
  {C}olloq., {N}ijmegen, 1975)}, pages 490--499. Lecture Notes in Phys., Vol.
  50. 1976.

\bibitem{KinElS84}
R.~C. King and N.~G.~I. El-Sharkaway.
\newblock Standard {Y}oung tableaux and character generators of classical {L}ie
  groups.
\newblock {\em J. Phys. A}, 17(1):19--45, 1984.

\bibitem{KoiTer90}
K.~Koike and I.~Terada.
\newblock Young diagrammatic methods for the restriction of representations of
  complex classical {L}ie groups to reductive subgroups of maximal rank.
\newblock {\em Adv. Math.}, 79(1):104--135, 1990.

\bibitem{Pro94}
R.~Proctor.
\newblock Young tableaux, {G}elfand patterns, and branching rules for classical
  groups.
\newblock {\em J. Algebra}, 164(2):299--360, 1994.

\bibitem{Sta99}
R.~Stanley.
\newblock {\em Enumerative combinatorics. {V}olume 2}.
\newblock Cambridge Studies in Advanced Mathematics~62. Cambridge University
  Press, Cambridge, 1999.

\end{thebibliography}
\bibliographystyle{plain}

\end{document}